\documentclass[11pt]{amsart}
\usepackage[dvipsnames]{xcolor}
\usepackage{amsfonts,amssymb,amsmath,amscd,amstext}
\usepackage{pgfplots}
\usepackage[colorlinks=true,linkcolor=teal,citecolor=purple]{hyperref}
\usepackage[utf8]{inputenc}
\usepackage{graphicx}
\usepackage{changes}
\usepackage{comment}
\usepackage{mathrsfs}
\usepackage{listings}
\usepackage[noabbrev,capitalize,nameinlink]{cleveref}
\usepackage{mathdots}
\usepackage{mathtools}
\usepackage[a4paper,top=2.1cm,bottom=2.1cm,left=1.7cm,right=1.7cm]{geometry}
\renewcommand{\leq}{\leqslant}
\renewcommand{\geq}{\geqslant}

\newcommand{\A}{\mathcal A}
\newcommand{\cA}{\A}
\newcommand{\F}{\mathcal F}
\newcommand{\ptl}{\partial}

\newcommand{\G}{\mathbb{G}}

\newcommand{\rb}{\partial^* E}

\newcommand{\dmc}{\mathrm{H}(E,\Om)}
\newcommand{\tmc}{|\mathrm{H}|(E,\Om)}

\newcommand{\rr}{{\mathbb{R}}}

\renewcommand{\k}{\kappa}

\newcommand{\Om}{\Omega}
\newcommand{\eps}{\varepsilon}

\newcommand{\ga}{\gamma}

\newcommand{\tv}{t_{\mathcal V}}

\newcommand{\mmc}{\mu^{\mathrm{H}}}

\newcommand{\perm}{P(E,\cdot)}

\newcommand{\ve}{\nu_E}

\newcommand{\Lag}{\mathscr{L}}
\newcommand{\cW}{\mathcal{W}_1}
\newcommand{\cV}{\mathcal{V}}
\renewcommand{\phi}{\varphi}
\definechangesauthor[color=red]{J}
\definecolor{champagne}{rgb}{0.97, 0.91, 0.81}

\definecolor{asparagus}{rgb}{0.53, 0.66, 0.42}

\DeclareMathOperator{\divv}{div}
\DeclareMathOperator{\dive}{div_{\partial E}}
\DeclareMathOperator{\trace}{trace}

\DeclareMathOperator{\supp}{supp}

\DeclareMathOperator{\arcsinh}{arcsinh}

\DeclareMathOperator{\jac}{Jac}

\newtheorem{theorem}{Theorem}[section]
\newtheorem{proposition}[theorem]{Proposition}
\newtheorem{definition}[theorem]{Definition}
\newtheorem{lemma}[theorem]{Lemma}
\newtheorem{corollary}[theorem]{Corollary}

\theoremstyle{definition}

\newtheorem{remark}[theorem]{Remark}
\newtheorem{example}[theorem]{Example}

\theoremstyle{remark}

\numberwithin{equation}{section}

\definechangesauthor[color=purple]{S}
\definechangesauthor[color=red]{P}

\title[Variational properties of the total inverse mean curvature]{Variational properties of the total inverse mean curvature in the plane under boundary constraints}

\author[J. Pozuelo]{Juli\'an Pozuelo} 
\address[Juli\'an Pozuelo]{Dipartimento di Matematica "Tullio Levi-Civita", Università degli Studi di Padova, via Trieste 63, 35131 Padova (PD), Italy}
\email{julian DOT pozuelodominguez AT unipd DOT it}

\author[S.~Verzellesi]{Simone Verzellesi}
\address[Simone Verzellesi]{Dipartimento di Matematica "Tullio Levi-Civita", Università degli Studi di Padova, via Trieste 63, 35131 Padova (PD), Italy}
\email{simone DOT verzellesi AT unipd DOT it}

\author[G.~Vianello]{Giacomo Vianello}
\address[Giacomo Vianello]{Dipartimento di Matematica "Tullio Levi-Civita", Università degli Studi di Padova, via Trieste 63, 35131 Padova (PD), Italy}
\email{giacomo DOT vianello AT unipd DOT it}

\date{\today}

\subjclass{Primary: 49Q10;\,53C24. Secondary: 53A10.}
\keywords{Heintze-Karcher inequality; total inverse mean curvature; boundary constraints.}
\thanks{\textit{Acknowledgments}.
	The authors thank M. Fogagnolo for stimulating conversations about the topics of the paper. S.V. and G.V. are members of Istituto Nazionale di Alta Matematica (INdAM - GNAMPA). S.V. and G.V. are supported by INdAM–GNAMPA 2025 Project \emph{Structure of sub-Riemannian hypersurfaces in Heisenberg groups}, 
	and by MIUR-PRIN 2022 Project \emph{Regularity problems in sub-Riemannian structures}}

\begin{document}

	\begin{abstract}
		We study the variational behavior of the total inverse mean curvature of curves with prescribed boundary in the half-plane. We characterize the existence of critical points with prescribed area. We show that such critical points are strongly stable. As an application, we prove a local minimality property. 
	\end{abstract}
	
	\maketitle

	\section{Introduction}
	
	The \emph{Heintze-Karcher inequality} in the Euclidean space states that, given a bounded domain $\Om\subseteq\rr^{n}$ with smooth boundary and with positive mean curvature $H$, it holds that
	\begin{equation}\label{HK-Euclidean}
		|\Om|\leq \frac{n-1}{n}\int_{\ptl\Om}\frac{1}{H}.
	\end{equation}
	Moreover, equality in \eqref{HK-Euclidean} holds if and only if $\Om$ is an Euclidean ball. Using dilations, \eqref{HK-Euclidean} is equivalent to the following optimization problem: spheres, and only spheres, minimize the total inverse mean curvature among closed, strictly mean-convex hypersurfaces enclosing the same volume. Inequality \eqref{HK-Euclidean} was first proved by Heintze and Karcher \cite{MR533065}. Later on, Ros \cite{MR996826} gave another proof of \eqref{HK-Euclidean} in spaces with nonnegative Ricci curvature by means of \emph{Reilly's formula} \cite{MR474149}. Owing to \eqref{HK-Euclidean}, he provided a new proof of \emph{Alexandrov theorem} \cite{MR86338,MR102114}, i.e. a closed, embedded, constant mean curvature hypersurface in $\rr^n$ is a round sphere, highlighting a link between the Heintze-Karcher inequality and the isoperimetric problem.\\
	
	Recently, various generalizations of the Heintze–Karcher inequality have been established. By means of the Aleksandrov-Bakelman-Pucci (ABP) method (cf. e.g. (\cite{MR1814364}), Xia and Zhang \cite{MR3702549} reproved \eqref{HK-Euclidean} in spaces of nonnegative Ricci curvature, while Brendle \cite{MR3090261} proved a Heintze-Karcher inequality in some \emph{warped products} to obtain an Aleksandrov theorem. Similar results have been achieved by Li and Xia \cite{MR4031740} in \emph{sub-static manifolds}, and have been recently improved by Borghini, Fogagnolo and Pinamonti \cite{MR4819146,MR4438902}. Other versions of \eqref{HK-Euclidean} have been obtained in \emph{weighted manifolds}  \cite{MR2161354} and in \emph{metric measure spaces} \cite{MR4127847}.\\
	
	A relevant proof of the Heintze-Karcher inequality was given by Montiel and Ros \cite{MR1173047} in space forms. Their approach, employing a clever fibration of $\Omega$ by geodesic segments, 
	provides another effective proof of the Alexandrov theorem. These ideas were adapted by Delgadino and Maggi \cite{MR3921314} to extend the Alexandrov theorem to the measure-theoretic framework: connected sets of finite perimeter with finite volume and constant distributional mean curvature are Euclidean balls. As a by-product, the authors extend \eqref{HK-Euclidean} to strictly mean convex sets in an appropriate \emph{viscosity} sense, e.g. allowing the presence of outward corners.\\
	
	
	Moreover, Montiel and Ros' argument paved the way for studying \emph{free-boundary} Heintze-Karcher inequalities and their applications in the setting of capillarity. Motivated by the study of the equilibrium shape assumed by a droplet of small volume on a flat homogeneous substrate, Delgadino and Weser \cite{MR4783570} obtained suitable Heintze-Karcher inequalities for smooth hypersurfaces with free-boundary in the half-space, and characterized spherical caps as the optimal configurations.
	A similar analysis has been recently performed in \cite{derosa2025rigiditycriticalpointshydrophobic} in a non-smooth setting.
	These examples (see also \cite{MR4562813,MR4817296,MR4567578}) further highlight the connection between the Heintze–Karcher inequality and the isoperimetric problem.\\

	In this paper, we show that the aforementioned symmetries between the the Heintze-Karcher inequality and the isoperimetric problem break down when we consider problems with \emph{constrained boundary}.
	Precisely, we consider the problem of minimizing the \emph{total inverse curvature}
	$$
	\mathcal F(\ga)=\int_\gamma\frac{1}{H},
	$$
	among positively curved curves $\ga:[0,2L]\to\rr^2_+$  with prescribed boundary $\{(-x_0,0)\}\cup\{(x_0,0)\}$ and enclosing a prescribed area. We call them \emph{admissible curves}. Here $\rr^2_+=\{(x,y): y>0\}$. Precisely, if $x_0,\A_0>0$ are fixed, denoting by $\A(\gamma)$ the area enclosed by $\gamma$ in $\rr^2_+$, our optimization problem reads as 
	\begin{equation}\label{intro_minimizationproblem}
		\inf\mathcal F(\ga),\qquad\gamma(0)=(x_0,0),\qquad\gamma(2L)=(-x_0,0)\qquad\A(\gamma)=\A_0.
	\end{equation}
	We propose a variational approach to the problem, with the aim of characterizing its equilibrium configurations and studying their stability properties. 
	The starting point is the identification of the correct notion of variation associated with \eqref{intro_minimizationproblem}. Accordingly, we say that a variation is \emph{admissible} when it induces a deformation of admissible curves. Among all admissible variations, the natural class to consider in connection with \eqref{intro_minimizationproblem} is that of \emph{area-preserving admissible variations}, namely those which preserves the enclosed area.
	Were \eqref{intro_minimizationproblem} consistent with the corresponding isoperimetric problem under the same boundary constraints - as it is in the free-boundary setting - we would expect its critical points to be one-dimensional spherical caps, i.e. circular arcs. We show that this is not the case. Indeed, stationary curves under area-preserving admissible variations are given explicitly by $\ga=(x,y):[0,2L]\to\rr^2$, where
	\begin{equation}\label{intro_criticalpoint}
		\begin{split}
			x(s)&=-\frac{\k}{2}\left(\frac{\sigma}{\pi+\sigma}\sin\left(\frac{\pi+\sigma}{\sigma}\arcsin\left(\frac{s-L}{\k}\right)\right)+\frac{\sigma}{\pi-\sigma}\sin\left(\frac{\pi-\sigma}{\sigma}\arcsin\left(\frac{s-L}{\k}\right)\right)\right)\\
			y(s)&=\frac{\k}{2}\left(\frac{\sigma}{\pi+\sigma}\cos\left(\frac{\pi+\sigma}{\sigma}\arcsin\left(\frac{s-L}{\k}\right)\right)+\frac{\sigma}{\pi-\sigma}\cos\left(\frac{\pi-\sigma}{\sigma}\arcsin\left(\frac{s-L}{\k}\right)\right)\right)+\frac{\pi x_0}{\sigma\tan\sigma}
		\end{split}
	\end{equation}
	Here $L=L(x_0,\A_0)$ is half the length of $\gamma$, and moreover
	\begin{equation*}
		\sigma=\pi\sqrt{\frac{x_0}{L+x_0}},\qquad \k=\frac{L}{ \sin\left(\sigma\right)}.
	\end{equation*}
	These configurations are uniquely determined by the data of the problem, namely $x_0$ and $\A_0$. Even more surprisingly, the existence of such configurations is governed by a specific compatibility condition among the parameters. In fact, the existence and uniqueness of such configurations are equivalent to the relation
	\begin{equation}\label{intro_treshold}
		\A_0>\frac{3}{2}\pi x_0^2.
	\end{equation}
	Heuristically, this threshold expresses the fact that equilibrium configurations cannot undergo excessive stretching.
	
		\begin{figure}[h!]
		\centering
		\begin{tikzpicture}[scale=1]
			\def\xo{1} 
			
			\begin{axis}[
				domain=0:8,
				samples=400,
				axis lines=middle,
				xlabel={$x$},
				ylabel={$y$},
				width=8cm,
				axis equal image, 
				xtick=\empty, ytick=\empty,
				]
				\pgfmathsetmacro{\L}{4}
				\pgfmathsetmacro{\k}{\L / (sin(deg(pi*sqrt(\xo/(\L+\xo)))))} 
				\pgfmathsetmacro{\G}{(1/pi)*sqrt((\L+\xo)/\xo)}
				
				\pgfmathsetmacro{\tmp}{(\k*\k)/(\L*\L) - 1}
				\pgfmathsetmacro{\tmppos}{max(0,\tmp)}
				\pgfmathsetmacro{\yshift}{sqrt(\xo*(\L+\xo)*\tmppos)}
				
				\addplot [
				domain=0:2*\L,
				samples=400,
				fill=red,
				fill opacity=0.25,
				draw=none
				]
				(
				{ -(\k/(2*\G*pi+2))*sin((\G*pi+1)*asin((x-\L)/\k))
					-(\k/(2*\G*pi-2))*sin((\G*pi-1)*asin((x-\L)/\k)) },
				{ (\k/(2*\G*pi+2))*cos((\G*pi+1)*asin((x-\L)/\k))
					+(\k/(2*\G*pi-2))*cos((\G*pi-1)*asin((x-\L)/\k))
					+\yshift }
				);
				
				\addplot [
				domain=0:2*\L,
				samples=400,
				thick,
				black
				]
				(
				{ -(\k/(2*\G*pi+2))*sin((\G*pi+1)*asin((x-\L)/\k))
					-(\k/(2*\G*pi-2))*sin((\G*pi-1)*asin((x-\L)/\k)) },
				{ (\k/(2*\G*pi+2))*cos((\G*pi+1)*asin((x-\L)/\k))
					+(\k/(2*\G*pi-2))*cos((\G*pi-1)*asin((x-\L)/\k))
					+\yshift }
				);
			\end{axis}
			
			\begin{axis}[
				at={(9cm,0cm)}, 
				domain=0:10,
				samples=400,
				axis lines=middle,
				xlabel={$x$},
				ylabel={$y$},
				width=9.2cm,
				axis equal image, 
				xtick=\empty, ytick=\empty,
				]
				\pgfmathsetmacro{\L}{5}
				\pgfmathsetmacro{\k}{\L / (sin(deg(pi*sqrt(\xo/(\L+\xo)))))} 
				\pgfmathsetmacro{\G}{(1/pi)*sqrt((\L+\xo)/\xo)}
				
				\pgfmathsetmacro{\tmp}{(\k*\k)/(\L*\L) - 1}
				\pgfmathsetmacro{\tmppos}{max(0,\tmp)}
				\pgfmathsetmacro{\yshift}{sqrt(\xo*(\L+\xo)*\tmppos)}
				
				\addplot [
				domain=0:2*\L,
				samples=400,
				fill=blue,
				fill opacity=0.25,
				draw=none
				]
				(
				{ -(\k/(2*\G*pi+2))*sin((\G*pi+1)*asin((x-\L)/\k))
					-(\k/(2*\G*pi-2))*sin((\G*pi-1)*asin((x-\L)/\k)) },
				{ (\k/(2*\G*pi+2))*cos((\G*pi+1)*asin((x-\L)/\k))
					+(\k/(2*\G*pi-2))*cos((\G*pi-1)*asin((x-\L)/\k))
					+\yshift }
				);
				
				\addplot [
				domain=0:2*\L,
				samples=400,
				thick,
				black
				]
				(
				{ -(\k/(2*\G*pi+2))*sin((\G*pi+1)*asin((x-\L)/\k))
					-(\k/(2*\G*pi-2))*sin((\G*pi-1)*asin((x-\L)/\k)) },
				{ (\k/(2*\G*pi+2))*cos((\G*pi+1)*asin((x-\L)/\k))
					+(\k/(2*\G*pi-2))*cos((\G*pi-1)*asin((x-\L)/\k))
					+\yshift }
				);
			\end{axis}
			
		\end{tikzpicture}
		\caption{Two equilibrium configurations with same boundary condition and different enclosed area. The configuration on the left-hand side is closer to the treshold \eqref{intro_treshold}}
		\label{fig:curve}
	\end{figure}
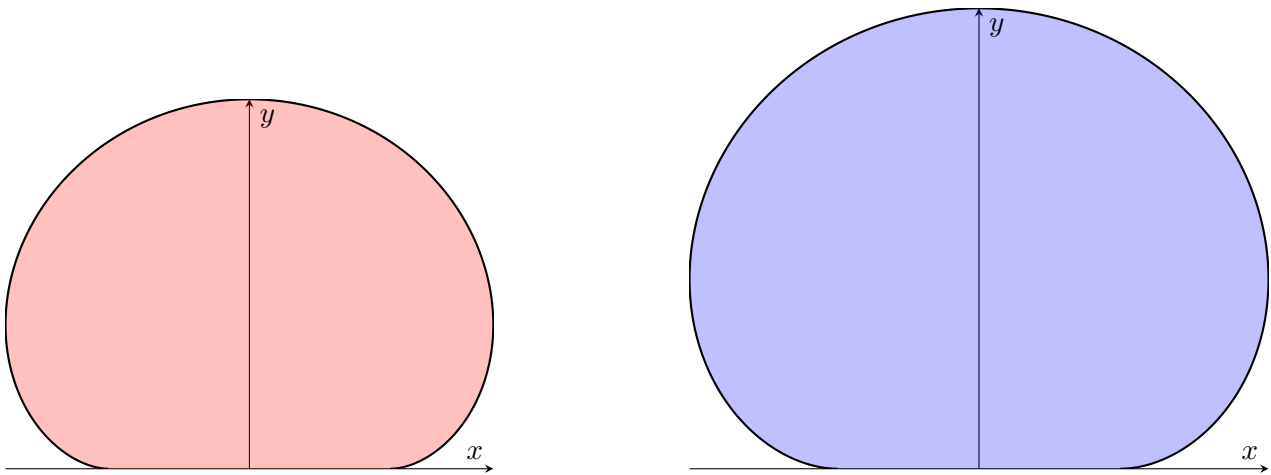
\bigskip
	 The characterization of critical points, as well as the derivation of \eqref{intro_treshold}, follows by imposing the vanishing of the first variation of $\F$, which we compute in \Cref{firstvariationofimc}, under area-preserving admissible variations. This leads to the \emph{Euler-Lagrange equation}
	\begin{equation}\label{intro_equazionepuntocriticoinlemma}
		2+\Delta^\gamma(H^{-2})=\lambda\qquad\text{ on $\gamma$},
	\end{equation}
	where $\Delta^\gamma$ is the Laplace-Beltrami operator on $\ga$ and $\lambda$ is the Lagrange multiplier emerging by the area constraint, and to the first-order boundary condition \begin{equation}\label{intr_firstorderbc}
		\gamma\text{ meets the horizontal axis tangentially}.
	\end{equation}
	It is worth noting that, although circular arcs solve \eqref{intro_equazionepuntocriticoinlemma} for $\lambda=2$, they are ruled out by the boundary tangency condition prescribed by \eqref{intr_firstorderbc}. 
	Moreover, the correct Lagrange multiplier is uniquely determined by the relation 
	\begin{equation}\label{intr_lagmolt}
		\lambda=\frac{2L}{L+x_0},
	\end{equation}
	and the first-order variational analysis of $\F$ is equivalent to that of the \emph{penalized functional}
	\begin{equation*}
		\F-\lambda\A
	\end{equation*}
	under arbitrary admissible variations, as we show in \Cref{uniquecriticalthm}. This fact is consistent with what happens to the $n$-dimensional area functional under volume constraints (cf. \cite{MR731682}). The characterization of critical points further exploits a Steiner-type symmetrization argument, which we explain in \Cref{sec_steiner}: optimal configurations must be symmetric with respect to the vertical axis. Our critical points yield a class of counterexamples to the validity of \eqref{HK-Euclidean} when inward corners are allowed. Indeed, the domain $\Om$ whose boundary consists of $\gamma(x_0,\A_0)$ and its reflection across the horizontal axis, satisfies
	\begin{equation*}
		\A(\Om)= \frac{1}{\lambda}\left(\frac{2(\pi^2-\sigma^2)\sigma+(\pi^2+\sigma^2)\sin 2\sigma}{2(\pi^2-\sigma^2)\sigma+(\pi^2-\sigma^2)\sin 2\sigma}\right)\int_{\ptl\bar\Om}\frac{1}{H}>\frac{1}{2}\int_{\ptl\Om}\frac{1}{H},
	\end{equation*}
	because \eqref{intr_lagmolt} implies that $\lambda<2$ (\Cref{rem_funzionalisupunticritici}). This fact prevents the generalization of \eqref{HK-Euclidean} to sets with strictly positive mean curvature everywhere but in a discrete set. 
	Subsequently, we address the problem of the stability of critical points, that is, the second-order properties of the total inverse curvature. The starting point is the computation of the second variations of $\F$ (\Cref{propsecvaroff}) and of its penalized version $\F-\lambda\A$ (\Cref{corollstbofg}), evaluated at the equilibrium configurations. Precisely, let $\mathcal V$ be any admissible variation of a fixed critical point. We denote by 
	\begin{equation*}
		X:=\left.\frac{\partial \mathcal V}{\partial t}\right|_{t=0}\qquad\text{and}  \qquad X':=\left.\frac{\partial^2 \mathcal V}{\partial t^2}\right|_{t=0}
	\end{equation*}
	its \emph{velocity vector} and \emph{acceleration vector} respectively. Assume that $\gamma$ is parametrized by arc-length. Let $\varphi,\,\varphi_\tau,\,\psi,\,\psi_\tau$ be such that 
	\begin{equation}\label{intro_quattrofunzioni}
		X=\varphi N+\varphi_\tau\dot\gamma\qquad\text{and}\qquad X'=\psi N+\psi_\tau\dot\gamma,
	\end{equation}
	where $N=\dot \ga^\perp$.  We show that the second variations of $\F$ and $\F-\lambda\A$ are given respectively by
	\begin{equation}\label{intro_secvarf}
		\left.\frac{d^2\F(\mathcal V(\cdot,t))}{dt^2}\right|_{t=0}=\int_0^{2L}\left(2H\varphi^2-\frac{2\dot\varphi^2}{H}+\frac{2\ddot\varphi^2}{H^3}\right)\,ds+\lambda\int_0^{2L}\left(\psi+H\varphi_\tau^2+2\varphi\dot\varphi_\tau\right)\,ds+\left[\frac{\dot\psi}{H^2}\right]_{0}^{2L}
	\end{equation}
	and
	\begin{equation}\label{intor_secvar2}
		\left.\frac{d^2(\F-\lambda \A)(\mathcal V(\cdot,t))}{dt^2}\right|_{t=0}=\int_0^{2L}\left((2-\lambda)H\varphi^2-\frac{2\dot\varphi^2}{H}+\frac{2\ddot\varphi^2}{H^3}\right)\,ds+\left[\frac{\dot\psi}{H^2}\right]_{0}^{2L}.
	\end{equation}
	Remarkably, the second variation of the penalized functional, apart from the boundary contribution, depends solely on the normal component of the velocity $X$, in contrast to what happens for $\F$. This highlights a further analogy with the area functional (cf. \cite{MR731682}). The correct notions of stability for $\F$ and for $\F-\lambda\A$ are then formulated by means of \eqref{intro_secvarf} and \eqref{intor_secvar2} (\Cref{stabilitydefinition}). Namely we impose, respectively, the non-negativity of \eqref{intro_secvarf} under area-preserving admissible variations, and of \eqref{intor_secvar2} under arbitrary admissible variations. Yet, as with the area functional (cf. \cite{MR731682}), the second-order behaviors of $\F$ and 
	$\F-\lambda\A$ are no longer equivalent. In fact, as we prove in \Cref{stabilityequivalent}, the stability of $\F$ under area-preserving variations coincides with the stability of $\F-\lambda\A$ under admissible variations satisfying 
	\begin{equation*}
		\int_0^{2L}\varphi\,ds=0,
	\end{equation*}
	which we call \emph{first-order area-preserving} admissible variations. As a byproduct of this equivalence, we deduce that the stability of $\F-\lambda\A$ under arbitrary variations is indeed a more restrictive condition. Nevertheless, we are able to show that the equilibrium configurations are stable in this strengthened sense (\Cref{teoremastabilitaperg}). This provides an interesting contrast with the behavior of the area functional under volume constraints, where circles are indeed \emph{unstable} critical points for the relevant penalized functional  (cf. \cite{MR731682}).
	By \eqref{intor_secvar2}, the study of the sign of the second variation of $\F-\lambda\A$ relies on the validity of a sharp weighted Hardy-type inequality of the form 
	\begin{equation}\label{intro_sharplowerbound}
		\mu_{\mathcal W_1}\int_0^{2L}\frac{2\dot \varphi^2}{H}\,ds\leq \int_0^{2L}\left(\frac{2\ddot \varphi^2}{H^3}+(2-\lambda)H \varphi^2\right)\,ds,
	\end{equation}
	for some $\mu_{\mathcal W_1}>0,$ where $\varphi$ belongs to the functional space
	\begin{equation*}
		\mathcal W_1=\left\{\varphi\in W^{2,2}(0,2L)\,:\,\varphi(0)=\varphi(2L)=0,\,\dot\varphi(0)=\dot\varphi(2L)=0\right\}
	\end{equation*} 
	We refer to \cite{Beesack_1971,Flork_Lasarz_2001,Kwong_Zettl_2006,Leighton_1970,Talenti_1975} for some accounts on weighted Hardy-type inequalities. The space $\mathcal W_1$ arises naturally imposing that variations are admissible. We prove the well-posedness of \eqref{intro_sharplowerbound} in \Cref{exminstab}, and we reduce the problem to determine the smallest positive constant $\mu$ for which the problem 
	\begin{equation}\label{intro_ELperstabilita}
		\left\{
		\begin{aligned}
			&\left(\frac{2}{H^3} \ddot \varphi\right)''+\mu \left(\frac{2}{H} \dot \varphi\right)' + (2-\lambda)H\varphi = 0\qquad\text{on }(0,2L), \\
			&\varphi(0) = \varphi(2L) = 0, \\
			&\dot \varphi (0) = \dot \varphi(2L) = 0
		\end{aligned}
		\right.
	\end{equation}
	admits a non-trivial solution (\Cref{characteigentwosided}). This equivalence follows from a variational analysis of \eqref{intro_sharplowerbound}, from which \eqref{intro_ELperstabilita} emerges as the relevant Euler–Lagrange equation. In \Cref{teoremastabilitaperg}, we show that $\mu_{\mathcal W_{1}}>1$. The proof of this fact relies on the explicit characterization of the general solution to the first equation in \eqref{intro_ELperstabilita} (\Cref{sec_soluzioniiiiiiiiiii}). In turn, the behavior of a solution depends on whether $\mu$
	is less than, equal to, or greater than the threshold value
	\begin{equation*}
		\frac{x_0}{L+x_0}+2\sqrt{\frac{x_0}{L+x_0}}.
	\end{equation*}
	Consequently, the analysis of \eqref{intro_ELperstabilita} must take these different cases into account. Combining \eqref{intor_secvar2} with \eqref{intro_sharplowerbound}, the fact that $\mu_{\mathcal W_1}>1$ and the fact that $H$ is strictly positive and bounded on $[0,2L]$, the second variation of $\F-\lambda\A$ admits the uniform lower bound 
	\begin{equation*}
		\left.\frac{d^2(\F-\lambda \A)(\mathcal V(\cdot,t))}{dt^2}\right|_{t=0}\geq  {\mathcal C}\|\varphi\|^2_{W^{2,2}(0,2L)},
	\end{equation*}
	for some $ C= C(x_0,\A_0)$. Our
	results can therefore be summarized in the following statement.  
	\begin{theorem}\label{intro_mainthm}
		Let $x_0,\A_0>0$ be fixed. The following facts hold.
		\begin{itemize}
			\item [(i)] $\mathcal F$ admits a critical point under area-preserving admissible variations if and only if \eqref{intro_treshold} holds.
			In this case, it is unique, and it is explicitly given by \eqref{intro_criticalpoint}. We denote it by $\gamma(x_0,\A_0)$.
			\item [(ii)] There exists $\tilde {\mathcal C}=\tilde {\mathcal C}(x_0,\A_0)$ such that, if $\mathcal V$ is an admissible variation of $\gamma(x_0,\A_0)$, then
			\begin{equation*}
				\left.\frac{d^2(\F-\lambda \A)(\mathcal V(\cdot,t))}{dt^2}\right|_{t=0}\geq  \tilde {\mathcal C}\|\varphi\|^2_{W^{2,2}(0,2L)},
			\end{equation*}
			where $\lambda$ is as in \eqref{intr_lagmolt} and $\varphi$ is as in \eqref{intro_quattrofunzioni}. In particular,  $\gamma(x_0,\A_0)$ is a stable critical point of $\F-\lambda \A$ under arbitrary admissible variations.
			\item [(iii)] If $\mathcal V$ is an area-preserving admissible variation of $\gamma(x_0,\A_0)$ and $\varphi$ is as above, then 
			\begin{equation*}
				\left.\frac{d^2\F(\mathcal V(\cdot,t))}{dt^2}\right|_{t=0}\geq  \tilde {\mathcal C}\|\varphi\|^2_{W^{2,2}(0,2L)}.
			\end{equation*}
			In particular,  $\gamma(x_0,\A_0)$ is a stable critical point of $\F$ under area-preserving admissible variations.
		\end{itemize}
	\end{theorem}
	Finally, as an application of \Cref{intro_mainthm}, we prove a local minimization property of equilibrium configurations when compared to sufficiently $C^2$-small normal graphs (\Cref{locminthm}).\\
	
	The paper is organized as follows. In \Cref{sec_preliminaries} we collect some preliminaries about planar curves (\Cref{subsec_Planarcurves}), and we provide the first variation formula for the total inverse mean curvature (\Cref{subsec_firstvarform}). In \Cref{sec_criticalpoints} we characterize critical points. Namely, we introduce admissible curves and variations (\Cref{sec_admissiblecurves}), we perform the Steiner's symmetrization argument (\Cref{sec_steiner}), we find out the equilibrium configurations (\Cref{sec_candidates}), we prove their stationarity (\Cref{sec_station}) and we prove the bijection between their length and their enclosed area (\Cref{sec_prescribedvolume}). In \Cref{sec_sectionstab} we compute the second variation formulas for $\F$ and $\F-\lambda\A$ (\Cref{sec_secvarform}), and we introduce and relate the relevant notions of stability (\Cref{sec_difnotstab}). In \Cref{sec_stability} we show the stability of critical points. Namely, we prove the existence of an optimal constant as in \eqref{intro_sharplowerbound} (\Cref{sec_exminprob}), we deduce \eqref{intro_ELperstabilita} (\Cref{sec_elforstab}), we find the general solution of the first equation in \eqref{intro_ELperstabilita} (\Cref{sec_soluzioniiiiiiiiiii}) and we show that $\mu_{\mathcal W_1}>1$ (\Cref{subsec:stab}). Finally, In \Cref{sec_locmin} we prove the local minimality property. 
	
	\section{Preliminaries}\label{sec_preliminaries}
	In the following, we let $\rr^2_+=\{(x,y)\in\rr^2\,:\,y>0\}$ and $\partial\rr^2_+=\{(x,y)\in\rr^2\,:\,y=0\}.$
	If $(\xi,\eta)\in\rr^2$, we set $(\xi,\eta)^\perp=(\eta,-\xi).$ If $a>0$ and $f:[0,a]\to \rr$, we may denote its derivative by $\dot f$, $f'$, $f^{(1)}$ or $\frac{df}{ds}$. We denote the right derivative of $f$ at $0$ by $\left.\frac{df}{ds}\right|_{s=0^+}$.
	\subsection{Planar curves}\label{subsec_Planarcurves}
	Let $L>0$ be fixed. We say that $\gamma=(x,y)\in C^1\left([0,2L],\rr^2\right)$ is a \emph{regular parametrization}, or a \emph{regular curve}, if $|\dot\gamma(s)|>0$ for any $s\in[0,2L]$. The support of a regular parametrization, $\gamma([0,2L])$, is an immersed curve of class $C^1$. Notice that $\gamma$ is an embedding if and only if it is injective. In the following, we shall identify a regular curve with its parametrization. A curve is \emph{parametrized by arc-length} if $|\dot\gamma|\equiv 1$. The tangent and normal bundles of $\gamma$ are spanned respectively by
	\begin{equation}\label{tangennnormmmm}
		\tau(s):=\frac{\dot\gamma(s)}{|\dot\gamma(s)|}\qquad\text{and}\qquad N(s):=\frac{\dot\gamma(s)^\perp}{|\dot\gamma(s)|}\qquad\text{for any $s\in[0,2L]$.}
	\end{equation}
	The \emph{curvature} of $\gamma$ is defined by
	\begin{equation*}  H(s):=\frac{\big\langle\dot\gamma(s),\ddot\gamma^\perp(s)\big\rangle}{|\dot\gamma(s)|^3}\qquad\text{for any $s\in[0,2L]$.}
	\end{equation*}
	We say that a parametrization is \emph{counterclockwise} if $H(s)\geq 0$ for any $s\in [0,2L]$, and that it is \emph{strictly counterclockwise} if $H(s)> 0$ for any $s\in [0,2L]$. We write $d\gamma$ to denote integration with respect to the $1$-dimensional Hausdorff measure. We recall that
	\begin{equation}\label{ibpformula}
		\int_\gamma \left(f\circ\gamma\right)\,d\gamma=\int_0^{2L}f(s)|\dot\gamma(s)|\,ds\qquad\text{for any $f\in C^1[0,2L]$. }
	\end{equation}
	The following lemma collects some formulas which will be needed in the following. 
	\begin{lemma}\label{lemmaconcontisugammadddot}
		Let $\gamma\in C^3([0,2L],\rr^2)$ be a regular curve parametrized by arc-length. Then, the following holds.
		\begin{equation}\label{dotgammaeddotgamma}
			\ddot \gamma=-H\dot\gamma^\perp\qquad\text{and}\qquad\dddot\gamma=-H^2\dot\gamma-\dot H\dot\gamma^\perp.
		\end{equation}
	\end{lemma}
	\begin{proof}
		As $|\dot\gamma|\equiv 1$, we recall that $\big\langle\dot\gamma,\ddot\gamma\big\rangle=0$ and $H=\big\langle\dot\gamma,\ddot\gamma^\perp\big\rangle$. In particular, $\ddot\gamma=-H\dot\gamma^\perp$. Moreover,
		\begin{equation*}
			\big\langle\dddot\gamma,\dot\gamma\big\rangle=\frac{d}{ds}\big\langle\ddot\gamma,\dot\gamma\big\rangle-|\ddot\gamma|^2=-|\ddot\gamma|^2=-H^2\qquad\text{and}\big\langle\dddot\gamma,\dot\gamma^\perp\big\rangle=\frac{d}{ds}\big\langle\ddot\gamma,\dot\gamma^\perp\big\rangle-\big\langle\ddot\gamma,\ddot\gamma^\perp\big\rangle=-\dot H,
		\end{equation*}
		whence \eqref{dotgammaeddotgamma} follows. 
	\end{proof}
	\subsection{A general first variation formula}\label{subsec_firstvarform}
	Let $S\subseteq\rr^n$ be a smooth, embedded, two-sided hypersurface with boundary $\partial S$. 
	We say that $S$ is \emph{strictly mean-convex} if $H>0$ on $S$, where $H$ is the non-averaged mean curvature of $S$. If $S$ is strictly mean-convex, we denote by
	\begin{equation*}
		\F(S)=\int_S\frac{1}{H}\,dS
	\end{equation*}
	the \emph{inverse total mean curvature} of $S$, where $dS$ denotes integration with respect to the $(n-1)$-dimensional Hausdorff measure. We denote by $\nabla^S$ and $\Delta^S$ the Levi-Civita connection and the Laplace-Beltrami operator of $S$, respectively. When $S$ encloses a bounded region $\Om$, we denote by $N$ its outer unit normal, and by $V(S)$ the Lebesgue measure of $|\Om|$. A smooth map $\mathcal V:\rr^n\times[-t_{\mathcal V},t_{\mathcal V}]\to\rr^n$
	is called a \emph{variation} of $S$ if it satisfies the following properties:
	\medskip
	\begin{enumerate}
		\item[(i)] $\mathcal V(p,0)=p$ for any $p\in S$;
		\medskip
		\item[(ii)] the set $S_t=\mathcal V(S,t)$ is a smooth, embedded, two-sided hypersurface for any $t\in [-t_{\mathcal V},t_{\mathcal V}]$;
		\medskip
		\item[(iii)] $\mathcal V(p,t)=p$ for any $p\in\partial S$ and any $t\in [-t_{\mathcal V},t_{\mathcal V}]$.
	\end{enumerate}
	Here, a variation fixes the boundary. For any $p\in S$ and any $t\in [-t_{\mathcal V},t_{\mathcal V}]$, we define 
	\begin{equation*}
		X(p)=\left.\frac{\partial \mathcal V(p,t)}{\partial t}\right|_{t=0}.
	\end{equation*}
	Clearly, $X\equiv 0$ on $\partial S$. If $S$ encloses a bounded region, then also $S_t$ encloses a bounded region for any sufficiently small $t$.  In this case, it is well-known (cf. \cite{MR4676392}) that 
	\begin{equation}\label{firstvarvol}
		\left. \frac{dV(S_t)}{dt}\right|_{t=0}=\int_S\langle X,N\rangle\,dS.
	\end{equation}
	In particular, when $V(S_t)\equiv V(S)$,
	\begin{equation}\label{vpthenfovp}
		\int_S\langle X,N\rangle\,dS=0.
	\end{equation}
	With respect to the first variation of $\F$, the following holds.
	\begin{proposition}\label{firstvariationofimc}
		Let $S\subseteq\rr^n$ be a smooth, embedded, two-sided, strictly mean-convex hypersurface with boundary $\partial S$. Let $\mathcal V$ be a variation of $S$. Then
		\begin{equation}\label{firsvarformulageneral}
			\left.\frac{d\F(S_t)}{dt}\right|_{t=0}=\int_S\left(\Delta^S(H^{-2})+\frac{|h|^2}{H^2}+1\right)\langle X,N\rangle\,dS+ \int_{\partial S}H^{-2}\langle\nabla^S\langle X,N\rangle,\eta\rangle\,d\partial S,
		\end{equation}
		where $d\partial S$ denotes integration with respect to the $(n-2)$-dimensional Hausdorff measure, $|h|$ is the norm of the second fundamental form of $S$ and $\eta$ is the conormal outward unit vector of $S$.
	\end{proposition}
	\begin{proof}
		We know by \cite[Lemma 1.26]{MR4676392} that 
		\begin{equation}\label{derivativeofHinfirstgeneral}
			\left. \frac{\partial H(\mathcal V(\cdot,t))}{\partial t}\right|_{t=0}=X^TH-\Delta^S\langle X,N\rangle-|h|^2\langle X,N\rangle,
		\end{equation}
		where $X^T=X-\langle X,N\rangle N$. Denote by $\jac(\mathcal V(\cdot,t)|_S)$ the Jacobian of $\mathcal V(\cdot,t)|_S$. Recall (cf. \cite{MR4676392}) that 
		\begin{equation}\label{derijacgenfirstvar}
			\left. \frac{\partial \jac(\mathcal V(\cdot,t)|_S)}{\partial t}\right|_{t=0}=\divv^S X,
		\end{equation}
		where $\divv ^S$ is the tangential divergence of $S$.
		The area formula and the divergence theorem yield
		\begin{equation*}
			\begin{split}
				\left.\frac{d\F(S_t)}{dt}\right|_{t=0}&= \left.\frac{d}{dt}\right|_{t=0}\int_S\frac{\jac(\mathcal V(\cdot,t)|_S)}{H(\mathcal V(\cdot,t))}\,dS\\
				&=\int_S\left(\left.\frac{\partial(H(\mathcal V(\cdot,t))^{-1})}{\partial t}\right|_{t=0}+\frac{1}{H}\left.\frac{\partial (\jac(\mathcal V(\cdot,t)|_S))}{\partial t}\right|_{t=0}\right)\,dS\\
				&\overset{\eqref{derijacgenfirstvar}}{=}
				\int_S\left(-\frac{1}{H^2}\left.\frac{\partial(H(\mathcal V(\cdot,t))}{\partial t}\right|_{t=0}+\frac{1}{H}\divv^S X\right)\,dS\\
				&\overset{\eqref{derivativeofHinfirstgeneral}}{=}
				\int_S\left(-\frac{1}{H^2}X^TH+\frac{1}{H^2}\Delta^S\langle X,N\rangle+\frac{|h|^2}{H^2}\langle X,N\rangle+\frac{1}{H}\divv^S X\right)\,dS\\
				&=
				\int_S\left(\divv^S\left(H^{-2}\nabla^S\langle X,N\rangle\right)-\langle\nabla^S(H^{-2}),\nabla^S\langle X,N\rangle\rangle+\frac{|h|^2}{H^2}\langle X,N\rangle\right)\,dS\\
				&\quad+
				\int_S\left(X^T(H^{-1})+H^{-1}\divv^S X^T+H^{-1}\divv^S \left(\langle X,N\rangle N\right)\right)\,dS\\
				&=
				\int_{\partial S}H^{-2}\langle\nabla^S\langle X,N\rangle,\eta\rangle\,d\partial S-\int_S\divv^S\left(\langle X,N\rangle\nabla^S(H^{-2})\right)\,dS\\
				&\qquad+\int_S\left(\Delta^S(H^{-2})+\frac{|h|^2}{H^2}+1\right)\langle X,N\rangle\,dS+\int_S\divv^S\left(H^{-1}X^T\right)\,dS\\
				&=\int_S\left(\Delta^S(H^{-2})+\frac{|h|^2}{H^2}+1\right)\langle X,N\rangle\,dS+ \int_{\partial S}H^{-2}\langle\nabla^S\langle X,N\rangle,\eta\rangle\,d\partial S. \qedhere
			\end{split}
		\end{equation*}
	\end{proof}
	
	\begin{remark}
		Recall that, when $n=2$, $|h|^2=H^2$. In particular, if $\gamma:[0,2L]\to\rr^2$ is a a smooth, embedded curve parametrized by arc-length and with $H>0$, then \eqref{firsvarformulageneral} reads as
		\begin{equation*}
			\left.\frac{d\F(\gamma_t)}{dt}\right|_{t=0}=\int_0^{2L}\left(\frac{d^2\left(H^{-2}\right)}{ds^2}+2\right)\langle X,N\rangle\,ds+ H^{-2}(2L)\left.\frac{d\langle X,N\rangle}{ds}\right|_{s=2L}-H^{-2}(0)\left.\frac{d\langle X,N\rangle}{ds}\right|_{s=0}.
		\end{equation*}
	\end{remark}
	\section{Critical points}\label{sec_criticalpoints}
	\subsection{Admissible curves and variations}\label{sec_admissiblecurves}
	Let $x_0>0$ and $L>0$ be fixed.
	We say that a curve $\gamma=(x,y)$ is \emph{admissible} if it satisfies the following properties:
	\medskip
	\begin{enumerate}
		\item [(i)]$\gamma\in C^\infty([0,2L],\rr^2)$ is a regular, strictly counterclockwise embedding;
		\medskip
		\item [(ii)]$\gamma(0)=(x_0,0)$ and $\gamma(2L)=(-x_0,0)$;
		\medskip
		\item [(iii)]$y(s)>0$ for any $s\in(0,2L)$.
	\end{enumerate}
	\medskip
	Any admissible curve encloses a bounded, simply connected, open region in $\rr^2_+$, that we denote by $\Omega$. The boundary of $\Omega$, $\partial\Om$, is the union of $\gamma$ with the segment $\left(-x_0,x_0\right)\times\{0\}$. As $\gamma$ is counterclockwise, then $N$ defined as in \eqref{tangennnormmmm} is the outer unit normal to $\Om\cap \rr^2_+$, and $H(s)$ is the curvature of $\partial\Om$ at $\gamma(s)$. We denote by $\A(\gamma)$ the area of $\Omega$. The divergence theorem implies that
	\begin{equation}\label{volumediomegaincurva}
		\A(\gamma)=\int_{0}^{2L}x(s)\dot y(s)\,ds=\frac{1}{2}\int_{0}^{2L}\left(x(s)\dot y(s)-\dot x(s) y(s)\right)\,ds.
	\end{equation}
	Finally, we say that an admissible curve is \emph{vertically symmetric} if it is symmetric with respect to the vertical axis.        
	The notion of admissible curve carries a natural notion of admissible variation. Precisely, a smooth map $\mathcal V:[0,2L]\times[-t_{\mathcal V},t_{\mathcal V}]\to\rr^2$ is called an \emph{admissible variation} of $\gamma$ if the following properties hold:
	\medskip
	\begin{enumerate}
		\item[(i)] $\mathcal V(\cdot,0)=\gamma$;
		\medskip
		\item[(ii)] $\gamma_t=(x_t,y_t):=\mathcal V(\cdot,t)$ is an admissible curve for any $t\in[-t_{\mathcal V},t_{\mathcal V}].$
		\medskip
	\end{enumerate}
	Let $\mathcal V(s,t)$ be an admissible variation.  We define the \emph{velocity} and \emph{acceleration} of $\mathcal V$ repectively by
	\begin{equation*}
		X(s)=\left.\frac{\partial \mathcal V(s,t)}{\partial t}\right|_{t=0},\qquad X'(s)=\left.\frac{\partial^2 \mathcal V(s,t)}{\partial t^2}\right|_{t=0} \qquad\text{for any $s\in[0,2L]$.}
	\end{equation*}
	We say that a variation is:
	\medskip
	\begin{enumerate}
		\item [(i)] \emph{normal} if $X(s)=\langle X(s),N(s)\rangle N(s)$ for any $s\in[0,2L]$;
		\medskip
		\item [(ii)]  \emph{geodesic} if $\mathcal V(s,t)=\gamma(s)+t X(s)$;
		\medskip
		\item [(iii)] \emph{area-preserving} if $\A(\gamma_t)=\A(\gamma)$ for any $t$;
		\medskip
		\item [(iv)] \emph{first-order area-preserving} if $\int_{\partial\Om}\langle X,N\rangle=0.$
		\medskip
	\end{enumerate}
	If in the above definitions the interval $[-t_{\mathcal V},t_{\mathcal V}]$ is replaced by $[0,t_{\mathcal V}]$, we call $\mathcal V$ a \emph{one-sided admissible variation}. One-sided variations are motivated by the fact that admissible curves must lie in the upper half-plane. 
	We introduce two relevant classes of admissible geodesic variations, which will suffice to characterize critical points of $\F$  under area-preserving variations. We need the following lemma.
	\begin{lemma}\label{lem}
		Let $\gamma$ be an admissible curve parametrized by arc-length. Let $t_{\mathcal V}>0$. Set
		\begin{equation}\label{variation}
			\mathcal V(s,t)=\gamma_t(s)=\gamma(s)+tX(s)\qquad\text{$s\in[0,2L]$, $t\in[-t_{\mathcal V},t_{\mathcal V}]$},
		\end{equation}
		where $X\in C^\infty([0,2L],R^2)$ satisfies $X(0)=X(2L)=0$. If $t_{\mathcal V}>0$ is sufficiently small, $\mathcal V$ is an admissible variation provided that $y_t>0$ on $(0,2L)$ for any $t\in [-t_{\mathcal V},t_{\mathcal V}]$.
	\end{lemma}
	\begin{proof}
		Notice that $\mathcal V(0,t)=(x_0,0)$ and $\mathcal V(2L,t)=(-x_0,0)$. Moreover,
		\begin{equation}\label{gammadotcomp}
			|\dot\gamma_t|^2=\langle\dot \gamma+t\dot X,\dot \gamma+t\dot X\rangle=1+2t\langle\dot X,\dot \gamma\rangle+t^2|\dot X|^2
		\end{equation}
		and
		\begin{equation}\label{bohnonmiricordo}
			\langle \dot\gamma_t,\ddot\gamma_t^\perp\rangle=\langle\dot \gamma+t\dot X,\ddot \gamma^\perp+t\ddot X^\perp\rangle=H+t\left(\langle\dot\gamma,\ddot X^\perp\rangle+\langle\ddot\gamma^\perp,\dot X\rangle\right)+t^2\langle\dot X,\ddot X^\perp\rangle.
		\end{equation}
		In particular, if $t$ is sufficiently small, $\gamma_t$ is a regular, smooth, strictly counterclockwise parametrization. 
		Since $\gamma$ is an embedding, then $\gamma_t$ is an embedding when $t$ is small. 
		Therefore, $\mathcal V$ is an admissible variation provided that $y_t>0$ on $(0,2L)$.
	\end{proof}
	
	In the above generality, it may happen that an arbitrary choice of $X$ leads to non-admissible variations. There are two relevant choices that lead to admissible variations.
	\begin{example}\label{comptsuppareadmissible}
		Choose $X\in C^\infty_c(0,2L)$. Then $\mathcal V$ is admissible. By the above considerations, it suffices to show that $y_t> 0$ on $(0,2L)$. If $X\equiv 0$ the claim is trivial. Assume $\|X\|_{L^\infty([0,2L],\rr^2)}>0$. Let $\eps>0$ be such that $\supp X\subseteq[\eps,2L-\eps]$. Since $y$ is continuous and positive on  $(0,2L)$, there exists $c_3>0$ such that $y(s)\geq c_3$ for any $s\in [\eps,2L-\eps]$. Therefore
		\begin{equation*}
			y_t=y+tX\geq c_3-|t|\|X\|_{L^\infty([0,2L],\rr^2)},
		\end{equation*}
		and the claim follows when $t$ is small.
	\end{example}
	\begin{example}\label{verticalvar}
		Choose $X\equiv(f,0)$ for any $f\in C^2[0,2L]$ such that $f(0)=f(2L)=0$. Then $y_t=y>0$ on $(0,2L)$, and $\mathcal V$ is trivially admissible. 
	\end{example}
	Following \cite{MR731682}, we can construct area-preserving variations sufficiently close to first-order area-preserving variations. 
	Although the following result is essentially \cite[Lemma 2.1]{MR731682}, we prove it for the sake of completeness.
	\begin{lemma}\label{docarmo}
		Fix $x_0>0$ and $L>0$. Let $\gamma:[0,2L]\to\rr^2$ be an admissible curve parametrized by arc-length. Let $\mathcal V$ be a first-order area-preserving admissible variation of $\gamma$. Denote by $X$ and $X'$ its velocity and acceleration respectively.  There exists an area-preserving admissible variation $ \mathcal {\tilde V}$, with velocity $\tilde X$ and acceleration $\tilde X'$, such that 
		\begin{equation}\label{docarmopropultimateequation}
			\tilde X=X\text{ on $[0,2L]$}\qquad\text{and}\qquad\tilde X'=X'\text{ on }\left[0,\frac{L}{2}\right]\cup\left[\frac{3L}{2},2L\right].
		\end{equation}
	\end{lemma}
	\begin{proof}
		Denote by $\Om$ the open set enclosed by $\gamma$.  Let $\hat\psi\in C^\infty_c\left[\frac{L}{2},\frac{3L}{2}\right]$ be such that 
		$\int_0^{2L}\hat\psi\,ds\neq 0$. For $t,u$ small, set $\mathcal V(\cdot,t,u)=\gamma+t X+u\hat\psi N$. Notice that
		\begin{equation*}
			\A\left(\mathcal V(\cdot,t,u)\right)=\frac{1}{2}\int_{0,2L}\left\langle\gamma+t X+u\hat\psi N,N+t\dot X^\perp+u\frac{d}{ds}\left(\hat\psi N^\perp\right)\right\rangle\,ds,
		\end{equation*}
		whence
		$\A\left(\mathcal V(\cdot,0,0)\right)=|\Om|$. Moreover, recalling that $X(0)=X(2L)=0$ and since $\mathcal V$ is first-order area-preserving,
		\begin{equation*}
			\left.\frac{\partial \A(\mathcal V(\cdot,t,u))}{\partial t}\right|_{(t,u)=(0,0)}=\int_0^{2L}\langle X,N\rangle\,ds=0  
		\end{equation*}
		and 
		\begin{equation*}
			\left.\frac{\partial \A(\mathcal V(\cdot,t,u))}{\partial u}\right|_{(t,u)=(0,0)}=\int_0^{2L}\hat \psi\,ds\neq 0.
		\end{equation*}
		Therefore, the implicit function theorem implies the existence of $\delta>0$ and $g\in C^\infty[-\delta,\delta]$ such that $g(0)=0$, $\dot g(0)=0$ and $\A(\mathcal V(\cdot,t,g(t)))=|\Om|$ for any $t\in[-\delta,\delta]$. Set $ \mathcal {\tilde V}(\cdot, t)=\gamma+tX+g(t)\hat\psi N$. By construction, $\mathcal {\tilde V}$ is area-preserving. Moreover, \eqref{docarmopropultimateequation} holds. Te proof that $ \mathcal {\tilde V}$ is admissible, up to choosing $t$ small, follows arguing as in \Cref{comptsuppareadmissible}.
	\end{proof}
	\subsection{Total inverse curvature under Steiner's symmetrization}\label{sec_steiner}
	In the following, we prove that we can reduce to consider vertically symmetric critical points. 
	Let $\gamma$ be an admissible curve. Denote by $H_\gamma(x,y)$ its curvature at $(x,y)$. We denote by $\Om$ the bounded domain enclosed by $\gamma$ and set $$\bar y=\max\{y\in\rr\,:\,(x,y)\in\overline \Om\}.$$ Since $\gamma$ is admissible, $\Om$ is convex. Therefore, there exist two functions $f,g\in C^{\infty}(0,\bar y)$ such that $f$ is convex, $g$ is concave and 
	\begin{equation*}
		\Om=\{(x,y)\in\rr^2\,f(y)< x< g(y),\,0<y<\bar y\}.
	\end{equation*}
	We define the \emph{Steiner's symmetrization} $\tilde\Om$ of $\Om$ by
	\begin{equation*}
		\tilde\Om=\left\{(x,y)\in\rr^2,\,-h(y)< x< h(y),\,0<y<\bar y\right\},\qquad\text{ where $h=\frac{g-f}{2}$. }
	\end{equation*}
	Notice that $-h$ is convex and $h$ is concave, whence $\tilde\Om$ is convex. Let us consider a regular parametrization $\tilde\gamma$ of $\partial\Om\cap\rr^2_+$. By construction, $\tilde\gamma$ is admissible provided that is strictly counterclockwise. 
	To this aim, let $(x,y)\in\gamma$. We assume first that $x>0$. Being $\tilde \gamma$ a graph over the vertical axis, 
	\begin{equation*}
		H_{\tilde \gamma}(x,y)=\frac{\Ddot f(y)-\Ddot g(y)}{2\left(1+\left(\frac{\Dot f(y)-\Dot g(y)}{2}\right)^2\right)^{\frac{3}{2}}}\geq\min\left\{  \frac{\Ddot f(y)}{(1+\Dot f(y)^2)^{\frac{3}{2}}},\frac{-\Ddot g(y)}{\left(1+\Dot g(y)^2\right)^{\frac{3}{2}}}\right\}=\min\{H_{\gamma}(-x,y),H_{\gamma}(x,y)\},
	\end{equation*}
	and in particular $H_{\tilde \gamma}(0,\bar y)\geq H_\gamma(0,\bar y)$.
	Arguing similarly when $x<0$, we conclude that $H_{\tilde\gamma}>0$, whence $\tilde\gamma$ is an admissible curve. Moreover, it is vertically symmetric by construction.
	We show that Steiner's symmetrization preserves the volume and decreases the total inverse curvature.
	
	\begin{proposition}
		Fix $x_0>0$. Let $\gamma$ be an admissible curve. Let $\tilde\gamma$ be its Steiner's symmetrization. Then $\A(\tilde\gamma)=\A(\gamma)$. Moreover, $\F(\tilde \gamma)< \F(\gamma)$ unless $\tilde \gamma=\gamma$.
	\end{proposition}
	\begin{proof}
		By Fubini's theorem, $\A(\tilde\gamma)=\A(\gamma)$. To conclude, we are left to show that, if $\tilde \gamma\neq \gamma,$ then
		\begin{equation*}
			\F(\tilde\gamma)=\int_0^{\bar y}\frac{4}{\Ddot f(y)-\Ddot g(y)}\left(1+\left(\frac{\Dot f(y)-\Dot g(y)}{2}\right)^2\right)^2\,d y<   \int_0^{\bar y}\left(\frac{(1+\Dot f(y)^2)^2}{\Ddot f(y)}-\frac{(1+\Dot g(y)^2)^2}{\Ddot g(y)}\right)\,dy=\F(\gamma).
		\end{equation*}
		To this aim, we consider the function $\Phi:\rr\times(0,+\infty)\longrightarrow\rr$ defined by
		\begin{equation*}
			\Phi(z,w)=\frac{(1+z^2)^2}{w}\qquad\text{ for any $(z,w)\in\rr\times(0,+\infty)$.}
		\end{equation*}
		A simple computation reveals that, for any $(z,w)\in \rr\times(0,+\infty)$, 
		\begin{equation*}
			\trace D^2\Phi(z,w)=\frac{4+12z^2}{w}+\frac{2(1+z^2)^2}{w^3}\qquad\text{and}\qquad\det D^2\Phi(z,w)=\frac{8(1+z^2)^3}{w^4}.
		\end{equation*}
		Since $w>0$, $D^2\Phi$ is positive definite, and therefore $\Phi$ is strictly convex. Hence, the thesis follows.
	\end{proof}
	\subsection{A family of candidate critical points}\label{sec_candidates}
	In this section we detect the candidate critical points. We begin with the following rigidity statement.
	\begin{proposition}\label{scoprirepunticritici}
		Let $x_0>0$ and $L>0$ be fixed. Then 
		there exist a unique $\lambda>0$ and a unique vertically symmetric admissible curve $\gamma=(x,y)$ parametrized by arc-length, such that $\dot y(0)=\dot y(2L)=0$ and        \begin{equation}\label{equazionepuntocriticoinlemma}
			2+\frac{d^2(H^{-2})}{ds^2}=\lambda\qquad\text{ on $[0,2L]$,}
		\end{equation} 
		if and only if  $3x_0<L$.
		In these cases,
		\begin{equation}\label{definitionoflambda}
			\lambda=\lambda(x_0,L)=\frac{2L}{L+x_0},
		\end{equation}
		and moreover
		\begin{align*}
			x(s)&=-\frac{\k}{2}\left(\frac{\sigma}{\pi+\sigma}\sin\left(\frac{\pi+\sigma}{\sigma}\arcsin\left(\frac{s-L}{\k}\right)\right)+\frac{\sigma}{\pi-\sigma}\sin\left(\frac{\pi-\sigma}{\sigma}\arcsin\left(\frac{s-L}{\k}\right)\right)\right)\\
			y(s)&=\frac{\k}{2}\left(\frac{\sigma}{\pi+\sigma}\cos\left(\frac{\pi+\sigma}{\sigma}\arcsin\left(\frac{s-L}{\k}\right)\right)+\frac{\sigma}{\pi-\sigma}\cos\left(\frac{\pi-\sigma}{\sigma}\arcsin\left(\frac{s-L}{\k}\right)\right)\right)+\frac{\pi x_0}{\sigma\tan\sigma}
		\end{align*}
		for any $s\in[0,2L]$, where 
		\begin{equation}\label{keg}
			\k=\frac{L}{ \sin\left(\pi\sqrt{\frac{x_0}{L+x_0}}\right)}\qquad\text{and}\qquad \sigma=\pi\sqrt{\frac{x_0}{L+x_0}}.
		\end{equation}
		In the following, we denote such admissible curves by $\gamma(x_0,L)$.
		\begin{proof}
			We divide the proof into some steps.\\
			\textbf{Step 1.} Assume that there exist $\lambda$ and $\gamma$ as in the statement. Let $A\in\rr$ be such that $2+2A=\lambda$. 
			Then 
			\begin{equation*}
				\frac{d^2}{ds^2}\left(\frac{1}{H^2}\right)=2A,
			\end{equation*}
			whence there exist $B,C\in\rr$ such that
			\begin{equation*}
				H(s)=\frac{1}{\sqrt{As^2+Bs+C}}\qquad\text{ for any $s\in [0,2L]$.}
			\end{equation*}  
			By assumption, $\gamma$ is vertically symmetric, so that     \begin{equation}\label{condizsimmetria}
				H(s)=H(2L-s)\qquad \text{for any $s\in[0,L]$}.
			\end{equation}   
			Notice that
			\begin{equation*}
				\begin{split}
					H(s)^{-2}-H(2L-s)^{-2}&=As^2+Bs-A(2L-s)^2-B(2L-s)\\
					&=A(4Ls-4L^2)+B(2s-2L)\\
					&=2AL(2s-2L)+B(2s-2L)\\
					&=(2s-2L)(2AL+B).
				\end{split}
			\end{equation*}
			Therefore, \eqref{condizsimmetria} holds if and only if $2AL+B=0$. In this case,
			\begin{equation*}
				H(s)=\frac{1}{\sqrt{As^2-2ALs+C}}.
			\end{equation*}
			Notice that $A\neq 0$,
			since otherwise $\gamma$ would be a part of a circle with center on the vertical axis, and in this case $\dot y(0)\neq 0$. Moreover, $C>0$, since $H(0)>0$.\\
			\textbf{Step 2.}
			We claim that $A<0$. Indeed, assume by contradiction that $A>0$. As $H(L)>0$, we have $        C-AL^2>0.$
			Notice that
			\begin{equation*}
				As^2-2ALs+C=A\left(s^2-2Ls+\frac{C}{A}\right)=A\left((s-L)^2+\frac{C-AL^2}{A}\right).
			\end{equation*}
			Since $\frac{C-AL^2}{A}>0$, there exists a unique $\k>0$ such that $\k^2=\frac{C-AL^2}{A}$. Conversely, given $A>0$, $L>0$ and $\k>0$, there exists a unique $C>0$ such that $\k^2=\frac{C-AL^2}{A}$. Therefore we may adopt the parameters $(\k,A,L)$ without loss of generality. In this way, 
			\begin{equation*}
				H(s)=\frac{1}{\sqrt{A}}\cdot\frac{1}{\sqrt{(s-L)^2+\k^2}}\qquad\text{  for any $s\in[0,2L]$. }
			\end{equation*}
			Set $$\theta(s)=\int_0^s H(l)\,dl.$$ Then 
			\begin{equation*}
				\theta(s)=\frac{1}{\sqrt{A}}\log\left(\frac{s-L+\sqrt{(s-L)^2+\k^2}}{-L+\sqrt{L^2+\k^2}}\right)\qquad\text{  for any $s\in[0,2L]$. }
			\end{equation*}
			As $\gamma$ is parametrized by arc-length, it is well-known that 
			\begin{equation}\label{itiswellknownthat}
				\dot x(s)=\cos\theta(s)\qquad\text{and}\qquad\dot y(s)=\sin\theta(s)\qquad\text{for any $s\in[0,2L]$.}
			\end{equation}
			Since $\gamma$ is vertically symmetric, then $\dot y(L)=0$, so that $\theta(L)=\pi$.
			%
			In particular, 
			\begin{equation*}
				\pi=\theta(L)=\frac{1}{\sqrt{A}}\log\left(\frac{k}{-L+\sqrt{L^2+k^2}}\right),
			\end{equation*}
		whence
		\begin{equation*}
			\begin{split}
				\theta(s)=\left(\frac{\log\left(s-L+\sqrt{(s-L)^2+\k^2}\right)-\log \left(-L+\sqrt{L^2+\k^2}\right)}{\log \k-\log\left(-L+\sqrt{L^2+\k^2}\right)}\right)\pi=\left(1+\frac{\arcsinh\left(\frac{s-L}{\k}\right)}{\arcsinh\left(\frac{L}{\k}\right)}\right)\pi
			\end{split}
		\end{equation*}
		for any $s\in[0,2L]$. Set $\tau=\arcsinh\left(\frac{L}{\k}\right).$ Then
		\begin{equation*}
			\cos\theta(s)=\cos\left(\pi+\frac{\pi}{\tau}\arcsinh\left(\frac{s-L}{\k}\right)\right)=-\cos\left(\frac{\pi}{\tau}\arcsinh\left(\frac{s-L}{\k}\right)\right).
		\end{equation*}
		Via the change of variable
		\begin{equation}\label{changevarpre}
			\frac{s-L}{\k}=\sinh l,\qquad ds=\k\cosh l,
		\end{equation}
		we infer that
		\begin{equation*}
			\begin{split}
				x(2L)-x(0)&=\int_0^{2L}\cos\theta(s)\,ds\\
				&\overset{\eqref{changevarpre}}{=}-\k\int_{-\tau}^{\tau}\cos\left(\frac{\pi}{\tau} l\right)\cosh l\,dl\\
				&=-\k \left[\frac{
					\frac{\pi}{\tau} \cosh(l) \sin(\frac{\pi}{\tau} l)
				}{
					1 + \frac{\pi^2}{\tau^2}
				}+\frac{\cos(\frac{\pi}{\tau} l) \sinh(l) }{1 + \frac{\pi^2}{\tau^2}}\right]_{-\tau}^{\tau}\\
				&=\frac{2L}{1 + \frac{\pi^2}{\tau^2}}.
		\end{split}
	\end{equation*}
	On the other hand, $x(2L)-x(0)=-2x_0<0$, a contradiction.\\
	\textbf{Step 3.}
	By the previous step, $A<0$. Since $C>0$, then $\frac{AL^2-C}{A}>0$, and there exists a unique $\k>L$ such that $\k^2=\frac{AL^2-C}{A}$. Conversely, given $A<0$, $L>0$ and $\k>L$, there exists a unique $C>0$ such that $\k^2=\frac{AL^2-C}{A}$. Arguing as above, we change parameters accordingly. Notice that
	\begin{equation*}
		As^2-2ALs+C=A\left((s-L)^2+\frac{C-AL^2}{A}\right)=(-A)\left(k^2-(s-L)^2\right)\qquad\text{ for any $s\in[0,2L]$.}
	\end{equation*}
	In this way, 
	\begin{equation*}
		H(s)=\frac{1}{\sqrt{-A}}\cdot\frac{1}{\sqrt{k^2-(s-L)^2}}\qquad\text{ for any $s\in[0,2L]$.}
	\end{equation*}
	Set again  $$\theta(s)=\int_0^s H(l)\,dl.$$ Then
	\begin{equation*}
		\theta(s)=\frac{1}{\sqrt{-A}}\arcsin\left(\frac{s-L}{k}\right)+\frac{1}{\sqrt{-A}}\arcsin\left(\frac{L}{k}\right)\qquad\text{ for any $s\in[0,2L]$.}
	\end{equation*}
	Arguing as above, $\theta(L)=\pi$, whence
	\begin{equation*}
		\frac{1}{\sqrt{-A}}=\frac{\pi}{\arcsin\left(\frac{L}{k}\right)}
	\end{equation*}
	and
	\begin{equation*}
		\theta(s)=\left(1+\frac{\arcsin\left(\frac{s-L}{k}\right)}{\arcsin\left(\frac{L}{k}\right)}\right)\pi\qquad\text{ for any $s\in[0,2L]$.}
	\end{equation*}
	Set $\sigma=\arcsin\left(\frac{L}{k}\right).$ Then
	\begin{equation*}
		\cos\theta(s)=\cos\left(\pi+\frac{\pi}{\sigma}\arcsin\left(\frac{s-L}{\k}\right)\right)=-\cos\left(\frac{\pi}{\sigma}\arcsin\left(\frac{s-L}{\k}\right)\right).
	\end{equation*}
	Since $\gamma$ is vertically symmetric, $x(L)=0$. Recalling \eqref{itiswellknownthat}, and via the change of variable
	\begin{equation}\label{changevar}
		\frac{s'-L}{\k}=\sin l,\qquad ds'=\k\cos l,
	\end{equation}
	we infer that
	\begin{equation*}
		\begin{split}
			x(s)&=\int_L^s\cos\theta(s')\,ds'\\
			&=-\k\int_0^{\arcsin\left(\frac{s-L}{\k}\right)}\cos \left(\frac{\pi}{\sigma} l\right)\cos l\,d l\\
			&=-\frac{\k}{2}\int_0^{\arcsin\left(\frac{s-L}{\k}\right)}\cos\left(\left(\frac{\pi}{\sigma}+1\right)l\right)\,dl-\frac{\k}{2}\int_0^{\arcsin\left(\frac{s-L}{\k}\right)}\cos\left(\left(\frac{\pi}{\sigma}-1\right)l\right)\,d l\\
			&=-\frac{\k}{2}\left[\frac{\sigma}{\pi+\sigma}\sin\left(\left(\frac{\pi+\sigma}{\sigma}\right)l\right)+\frac{\sigma}{\pi-\sigma}\sin\left(\left(\frac{\pi-\sigma}{\sigma}\right)l\right)\right]_0^{\arcsin\left(\frac{s-L}{\k}\right)}\\
			&-\frac{\k}{2}\left(\frac{\sigma}{\pi+\sigma}\sin\left(\frac{\pi+\sigma}{\sigma}\arcsin\left(\frac{s-L}{\k}\right)\right)+\frac{\sigma}{\pi-\sigma}\sin\left(\frac{\pi-\sigma}{\sigma}\arcsin\left(\frac{s-L}{\k}\right)\right)\right).
		\end{split}
	\end{equation*}
	Therefore, imposing the boundary condition $x(0)=x_0$ and noticing that
	\begin{align*}
		\sin\left(\frac{\pi+\sigma}{\sigma}\arcsin\left(\frac{-L}{\k}\right)\right)&=-\sin\left(\pi+\sigma\right)=\frac{L}{\k},\\
		\sin\left(\frac{\pi-\sigma}{\sigma}\arcsin\left(\frac{-L}{\k}\right)\right)&=-\sin\left(\pi-\sigma\right)=-\frac{L}{\k},
	\end{align*}
	we deduce that
	\begin{equation}\label{pertrovkinok}
		x_0=x(0)=\frac{L\sigma^2}{\pi^2-\sigma^2}.
	\end{equation}
	Since $0<L<k,$  \eqref{pertrovkinok} is equivalent to
	\begin{equation*}
		\arcsin\left(\frac{L}{\k}\right)=\pi\sqrt{\frac{x_0}{ L+x_0 }},
	\end{equation*}
	which in turn admits a solution if and only if $3x_0<L$. 
	In such cases, 
	\begin{equation*}
		\k=\frac{L}{ \sin\left(\pi\sqrt{\frac{x_0}{L+x_0}}\right)},\qquad \sigma=\pi\sqrt{\frac{x_0}{L+x_0}}\qquad\text{and}\qquad    \lambda=\frac{2L}{L+x_0}.
	\end{equation*}
	We proved that, assuming that $0<3x_0<L$, there exists at most one couple $\gamma$ and $\lambda$ as in the statement.\\
	\textbf{Step 4.} To prove existence it suffices to verify that the curve of the previous step is a solution. To this aim, we just need to find $y$. Notice that 
	\begin{equation*}
		\sin\theta(s)=\sin\left(\pi+\frac{\pi}{\sigma}\arcsin\left(\frac{s-L}{\k}\right)\right)=-\sin\left(\frac{\pi}{\sigma}\arcsin\left(\frac{s-L}{\k}\right)\right).
	\end{equation*}
	Changing variables as in \eqref{changevar}, and recalling that $y(0)=0$,
	\begin{equation*}
		\begin{split}
			y(s)&=\int_0^s\sin\theta(s')\,ds'\\
			&=-\k\int^{\arcsin\left(\frac{s-L}{\k}\right)}_{-\sigma}\sin \left(\frac{\pi}{\sigma}l\right)\cos l\,dl\\
			&=-\frac{\k}{2}\int^{\arcsin\left(\frac{s-L}{\k}\right)}_{-\sigma}\sin\left(\left(\frac{\pi}{\sigma}+1\right)l\right)\,d l-\frac{\k}{2}\int^{\arcsin\left(\frac{s-L}{\k}\right)}_{-\sigma}\sin\left(\left(\frac{\pi}{\sigma}-1\right)l\right)\,d l\\
			&=\frac{\k}{2}\left[\frac{\sigma}{\pi+\sigma}\cos\left(\left(\frac{\pi+\sigma}{\sigma}\right)l\right)+\frac{\sigma}{\pi-\sigma}\cos\left(\left(\frac{\pi-\sigma}{\sigma}\right)l\right)\right]_{-\sigma}^{\arcsin\left(\frac{s-L}{\k}\right)}\\
			&=\frac{\k}{2}\left(\frac{\sigma}{\pi+\sigma}\cos\left(\frac{\pi+\sigma}{\sigma}\arcsin\left(\frac{s-L}{\k}\right)\right)+\frac{\sigma}{\pi-\sigma}\cos\left(\frac{\pi-\sigma}{\sigma}\arcsin\left(\frac{s-L}{\k}\right)\right)\right)+\frac{\pi x_0}{\sigma\tan\sigma}.
		\end{split}
	\end{equation*}
	The remaining properties are straightforward, and the thesis follows.
\end{proof}
\end{proposition}
\subsection{Two equivalent notions of stationarity}\label{sec_station}
Fix $x_0>0$ and $L>0$. Let $\gamma$ be an admissible curve. We recall that 
\begin{equation*}
\F(\gamma)=\int_\gamma\frac{1}{H}\,d\gamma.
\end{equation*}
For any $\lambda\in\rr$, we also introduce the penalized functional 
\begin{equation*}
\F(\gamma)-\lambda \A(\gamma).
\end{equation*}
\begin{definition}
Fix $x_0>0$. Let $\gamma$ be an admissible curve.

\begin{enumerate}
	\item We say that $\gamma$ is an \emph{area-preserving critical point} of $\F$ if
	\begin{equation*}
		\left.\frac{d\F(\mathcal V(\cdot,t))}{dt}\right|_{t=0}=0\qquad(\text{respectively $\geq 0$})
	\end{equation*}
	for any area-preserving admissible (respectively one-sided admissible) variation $\mathcal V$ of $\gamma$.
	\item We say that $\gamma$ is a \emph{critical point} of $\F-\lambda\A$ if
	\begin{equation*}
		\left.\frac{d(\F-\lambda\A)(\mathcal V(\cdot,t))}{dt}\right|_{t=0}=0\qquad(\text{respectively $\geq 0$})
	\end{equation*}
	for any admissible (respectively one-sided admissible) variation $\mathcal V$ of $\gamma$.
\end{enumerate}
\end{definition}
The next result shows that the admissible curves of \Cref{scoprirepunticritici} are the unique critical points of both problems.
\begin{theorem}\label{uniquecriticalthm}
Fix $x_0>0$ and $L>0$. The following holds.
\begin{enumerate}
	\item [(i)]$\F$ admits a vertically symmetric area-preserving critical point $\gamma$ of length $2L$ if and only if $3x_0<L$. In these cases, $\gamma=\gamma(x_0,L)$ and it is unique.
	\item [(ii)]$\F-\lambda\A$ admits a vertically symmetric  critical point $\gamma$ of length $2L$ if and only if $3x_0<L$ and $\lambda=\lambda(x_0,L)$. In these cases, $\gamma=\gamma(x_0,L)$ and it is unique.
\end{enumerate}
\end{theorem}
\begin{proof}
We start proving (i). Assume that $3x_0<L$. Then we can consider the curve $\gamma(x_0,L)$ given by \Cref{scoprirepunticritici}. Let $\mathcal V(\cdot,t)=\gamma_t=(x_t,y_t)$ be an area-preserving admissible variation of $\gamma(x_0,L)$. By \Cref{firstvariationofimc} and \Cref{scoprirepunticritici}, it holds that
\begin{equation*}
	\begin{split}
		\left.\frac{d\F(\mathcal V(\cdot,t))}{dt}\right|_{t=0}&=\int_0^{2L}\left(2+\frac{d^2(H^{-2})}{ds^2}\right)\langle X,N\rangle\,ds+ H^{-2}(2L)\left.\frac{d\langle X,N\rangle}{ds}\right|_{s=2L}-H^{-2}(0)\left.\frac{d\langle X,N\rangle}{ds}\right|_{s=0}\\
		&        =\lambda(x_0,L)\int_0^{2L}\langle X,N\rangle\,ds+ H^{-2}(2L)\left(\left.\frac{d\langle X,N\rangle}{ds}\right|_{s=2L}-\left.\frac{d\langle X,N\rangle}{ds}\right|_{s=0}\right)\\
		&\overset{\eqref{vpthenfovp}}{=} H^{-2}(2L)\left(\left.\frac{d\langle X,N\rangle}{ds}\right|_{s=2L}-\left.\frac{d\langle X,N\rangle}{ds}\right|_{s=0}\right)\\
		&\overset{X(0)=X(2L)=0}{=} H^{-2}(2L)\left(\left\langle\dot X(2L),N(2L)\right\rangle-\left\langle\dot X(0),N(0)\right\rangle\right)\\
		&\overset{\dot y(0)=\dot y(2L)=0}{=} H^{-2}(2L)\left(\left.\frac{\partial^2y_t(s)}{\partial t\partial s}\right|_{(t,s)=(0,0)}-\left.\frac{\partial^2y_t(s)}{\partial t\partial s}\right|_{(t,s)=(0,2L)}\right).
	\end{split} 
\end{equation*}
Notice that $y_t(0)=y_t(2L)=0$ and $y_t(s)>0$ for any $s\in[0,2L]$ and any $t\in[-t_{\mathcal V},t_{\mathcal V}]$. Therefore $\dot y_t(0)\geq 0$ and $\dot y_t(2L)\leq 0$ for any $t\in[-t_{\mathcal V},t_{\mathcal V}]$. Since $\dot y_0(0)=\dot y_0(2L)=0$, we conclude that 
\begin{equation*}
	\left.\frac{\partial^2y_t(s)}{\partial t\partial s}\right|_{(t,s)=(0,0)}=\left.\frac{\partial^2y_t(s)}{\partial t\partial s}\right|_{(t,s)=(0,2L)}=0,
\end{equation*}
whence 
\begin{equation*}
	\left.\frac{d\F(\mathcal V(\cdot,t))}{dt}\right|_{t=0}=0.
\end{equation*}
If $\mathcal V$ is just an area-preserving one-sided admissible variation, we proceed as above, noticing that in this case we can just conclude that 
\begin{equation*}
	\left.\frac{\partial^2y_t(s)}{\partial t\partial s}\right|_{(t,s)=(0,0)}\geq 0\qquad\text{and}\qquad\left.\frac{\partial^2y_t(s)}{\partial t\partial s}\right|_{(t,s)=(0,2L)}\leq 0,
\end{equation*}
so that 
\begin{equation*}
	\left.\frac{d\F(\mathcal V(\cdot,t))}{dt}\right|_{t=0}\geq 0.
\end{equation*}
Therefore, $\gamma(x_0,L)$ is an area-preserving critical point of $\F$. To show uniqueness, let $\ga$ be a vertically symmetric area-preserving critical point of length $2L$. Let $\varphi\in C^\infty_c(0,2L)$ be such that $\int_0^{2L}\varphi\,ds=0$. By \Cref{docarmo} and \Cref{comptsuppareadmissible}, we find an area-preserving admissible variation such that $X=\varphi N$. Then
\begin{equation}\label{intinproofcriticalpoints}
	0=  \left.\frac{d\F(\mathcal V(\cdot,t))}{dt}\right|_{t=0}=\int_0^{2L}\left(2+\frac{d^2(H^{-2})}{ds^2}\right)\varphi\,ds.
\end{equation}
By \eqref{intinproofcriticalpoints}, there exists $\lambda\in\rr$ such that
\begin{equation*}
	2+\frac{d^2(H^{-2})}{ds^2}=\lambda
\end{equation*}
on $[0,2L]$. Let $f\in C^2[0,2L]$ be such that $f(0)=0$, $f(2L)=0$, $\dot f(0)=1$ and $\dot f(2L)=0$. By \Cref{docarmo} and \Cref{verticalvar}, there exists an area-preserving admissible variation such that $X=(f,0)$. Then, arguing as above and by \eqref{intinproofcriticalpoints}, \begin{equation*}
	0=  \left.\frac{d\F(\mathcal V(\cdot,t))}{dt}\right|_{t=0}=H^{-2}(2L)\left(\left\langle\dot X(2L),N(2L)\right\rangle-\left\langle\dot X(0),N(0)\right\rangle\right)=-\dot y(0).
\end{equation*}
Therefore $\dot y(0)=0$. Similarly, $\dot y(2L)=0$. By \Cref{scoprirepunticritici}, we conclude that $\gamma=\gamma(x_0,L)$ and that 
$3x_0<L$ is necessary, and (i) follows. The proof of (ii) is a straightforward adaptation of the above.
\end{proof}
\subsection{Uniqueness of critical points with prescribed area}\label{sec_prescribedvolume}
The statement of \Cref{uniquecriticalthm} ensures existence of critical points with prescribed length $2L$, provided that $0<3x_0<L$. In this subsection, we prove that it is equivalent to prescribe the enclosed area. Indeed, for critical points, these two quantities are in bijection, as the next proposition shows.
\begin{proposition}\label{volbijlength}
Fix $0<3x_0<L$. Let $\A_0(x_0,L)$ be the area enclosed by $\gamma(x_0,L)$.   Then
\begin{equation}\label{volumeofcriticalpointsinstatement}
	\A_0(x_0,L)=\frac{ L+x_0 }{2\pi}\left(\frac{\pi\sqrt{\frac{x_0}{ L+x_0 }}}{\sin\left(\pi\sqrt{\frac{x_0}{ L+x_0 }}\right)}\right)^2\left( L+(2 L+x_0 )\frac{\sin\left(2\pi\sqrt{\frac{x_0}{ L+x_0 }}\right)}{2\pi\sqrt{\frac{x_0}{ L+x_0 }}}\right).
\end{equation}
In particular, 
\begin{equation*}
	\frac{\partial \A_0(x_0,L)}{\partial L}>0, \qquad\inf_{L>3x_0}\A_0(x_0,L)=\lim_{L\to 3x_0}\A_0(x_0,L)=\frac{3}{2}\pi x_0^2.
\end{equation*}
In conclusion, for any $\A_0>\frac{3}{2}\pi x_0^2$, there exists a unique $L=L(x_0,\A_0)>3x_0$ such that $\A_0=\A_0(x_0,L)$.
\end{proposition}

\begin{proof}
First we prove \eqref{volumeofcriticalpointsinstatement}.
\begin{equation*}
	\begin{split}
		\A_0(x_0,L)&\overset{\eqref{volumediomegaincurva}}{=}\int_0^{2L} x(s)\dot y(s)\,ds\\
		&=-\frac{\k\sigma}{2\pi+2\sigma}\int_0^{2L}\sin\left(\frac{\pi+\sigma}{\sigma}\arcsin\left(\frac{s-L}{\k}\right)\right)\sin\left(\pi+\frac{\pi}{\sigma}\arcsin\left(\frac{s-L}{\k}\right)\right)\,ds\\
		&\quad-\frac{\k\sigma}{2\pi-2\sigma}\int_0^{2L}\sin\left(\frac{\pi-\sigma}{\sigma}\arcsin\left(\frac{s-L}{\k}\right)\right)\sin\left(\pi+\frac{\pi}{\sigma}\arcsin\left(\frac{s-L}{\k}\right)\right)\,ds\\
		&=\frac{\k\sigma}{2\pi+2\sigma}\int_0^{2L}\sin\left(\frac{\pi+\sigma}{\sigma}\arcsin\left(\frac{s-L}{\k}\right)\right)\sin\left(\frac{\pi}{\sigma}\arcsin\left(\frac{s-L}{\k}\right)\right)\,ds\\
		&\quad+\frac{\k\sigma}{2\pi-2\sigma}\int_0^{2L}\sin\left(\frac{\pi-\sigma}{\sigma}\arcsin\left(\frac{s-L}{\k}\right)\right)\sin\left(\frac{\pi}{\sigma}\arcsin\left(\frac{s-L}{\k}\right)\right)\,ds,\\
	\end{split}
\end{equation*}
where $\sigma$ is defined in \eqref{keg}. We change variables as in \eqref{changevar}, and obtain that
\begin{equation*}
	\begin{split}
		\A_0(x_0,L)&=\frac{\k^2\sigma}{2\pi+2\sigma}\int_{-\sigma}^\sigma   \sin\left(\left(\frac{\pi+\sigma}{\sigma}\right)l\right)\sin\left(\frac{\pi}{\sigma}l\right)\cos l\,dl\\
		&\quad+\frac{\k^2\sigma}{2\pi-2\sigma}\int_{-\sigma}^\sigma   \sin\left(\left(\frac{\pi-\sigma}{\sigma}\right)l\right)\sin\left(\frac{\pi}{\sigma}l\right)\cos l\,dl\\
		&=\frac{\k^2\sigma}{4\pi+4\sigma}\left(\int_{-\sigma}^\sigma  \sin^2\left(\left(\frac{\pi+\sigma}{\sigma}\right)l\right)\,dl+\int_{-\sigma}^\sigma   \sin\left(\left(\frac{\pi+\sigma}{\sigma}\right)l\right)\sin\left(\left(\frac{\pi-\sigma}{\sigma}\right)l\right)\,dl\right)\\
		&\quad+\frac{\k^2\sigma}{4\pi-4\sigma}\left(\int_{-\sigma}^\sigma  \sin^2\left(\left(\frac{\pi-\sigma}{\sigma}\right)l\right)\,dl+\int_{-\sigma}^\sigma   \sin\left(\left(\frac{\pi+\sigma}{\sigma}\right)l\right)\sin\left(\left(\frac{\pi-\sigma}{\sigma}\right)l\right)\,dl\right)\\
		&=\frac{\k^2\sigma^2}{2(\pi+\sigma)^2}\int_{0}^{\pi+\sigma}\sin^2\left(\zeta\right)\,d\zeta+\frac{\k^2\sigma^2}{2(\pi-\sigma)^2}\int_{0}^{\pi-\sigma}\sin^2\left(\zeta\right)\,d\zeta\\
		&\quad+\frac{\pi \k^2\sigma}{2(\pi^2-\sigma^2)}\int_{-\sigma}^\sigma   \sin\left(\left(\frac{\pi+\sigma}{\sigma}\right)l\right)\sin\left(\left(\frac{\pi-\sigma}{\sigma}\right)l\right)\,dl.
	\end{split}
\end{equation*}
First,
\begin{equation*}
	\begin{split}
		\frac{\k^2\sigma^2}{2(\pi+\sigma)^2}\int_{0}^{\pi+\sigma}\sin^2\left(\zeta\right)\,d\zeta     
		&=\frac{\k^2\sigma^2}{4\pi+4\sigma}-\frac{k^2\sigma^2}{8(\pi+\sigma)^2}\sin2\sigma\\
	\end{split}
\end{equation*}
Similarly,
\begin{equation*}
	\begin{split}
		\frac{\k^2\sigma^2}{2(\pi-\sigma)^2}\int_{0}^{\pi-\sigma}\sin^2\left(\zeta\right)\,d\zeta     
		&=\frac{\k^2\sigma^2}{4\pi-4\sigma}+\frac{k^2\sigma^2}{8(\pi-\sigma)^2}\sin2\sigma\\
	\end{split}
\end{equation*}
Finally,
\begin{equation*}
	\begin{split}
		\frac{\pi \k^2\sigma}{2(\pi^2-\sigma^2)}\int_{-\sigma}^\sigma   &\sin\left(\left(\frac{\pi+\sigma}{\sigma}\right)l\right)\sin\left(\left(\frac{\pi-\sigma}{\sigma}\right)l\right)\,dl\\
		&=   \frac{\pi \k^2\sigma}{4(\pi^2-\sigma^2)}\int_{-\sigma}^\sigma\cos 2l\,dl
		- \frac{\pi \k^2\sigma}{4(\pi^2-\sigma^2)}\int_{-\sigma}^\sigma\cos \left(\frac{2\pi}{\sigma}l\right)\,dl\\
		&= \frac{\pi \k^2\sigma}{4(\pi^2-\sigma^2)}\sin2\sigma.
	\end{split}
\end{equation*}
In conclusion,
\begin{equation}\label{computvolume}
	\begin{split}
		\A_0(x_0,L)
		&=\frac{ \pi \k^2\sigma^2}{2(\pi^2-\sigma^2)}+\frac{\pi \k^2\sigma(\pi^2+\sigma^2)}{4(\pi^2-\sigma^2)^2}\sin2\sigma\\
		&=\frac{ L+x_0 }{2\pi}\left(\frac{\pi\sqrt{\frac{x_0}{ L+x_0 }}}{\sin\left(\pi\sqrt{\frac{x_0}{ L+x_0 }}\right)}\right)^2\left( L+(2 L+x_0 )\frac{\sin\left(2\pi\sqrt{\frac{x_0}{ L+x_0 }}\right)}{2\pi\sqrt{\frac{x_0}{ L+x_0 }}}\right).
	\end{split}
\end{equation}
In particular,
\begin{equation*}
	\frac{\partial \A_0(x_0,L)}{\partial L}=\frac{\left(
		4 L \pi^2 \frac{\cos\!\left(\pi \sqrt{\tfrac{x_0}{L + x_0}}\right)}{\sin\!\left(\pi \sqrt{\tfrac{x_0}{L + x_0}}\right)}
		+ (3L + 4x_0)\left( \left(\frac{ L+x_0 }{x_0}\right)\sin\!\left(2 \pi \sqrt{\tfrac{x_0}{L + x_0}}\right) 
		+ 2 \pi \sqrt{\frac{ L+x_0 }{x_0}} \right)
		\right)}
	{8\left(\frac{L+x_0}{x_0}\right)^{\frac{3}{2}}\, \sin^2\!\left(\pi \sqrt{\tfrac{x_0}{L + x_0}}\right)}.
\end{equation*}
By $L>3x_0$, $\sin\left(\pi\sqrt{\frac{x_0}{ L+x_0 }}\right)>0,$ $\cos\left(\pi\sqrt{\frac{x_0}{ L+x_0 }}\right)>0$ and $\sin\left(2\pi\sqrt{\frac{x_0}{ L+x_0 }}\right)>0$, whence $\frac{\partial \A_0(x_0,L)}{\partial L}>0$. Therefore, $\inf_{L>3x_0}\A_0(x_0,L)=\lim_{L\to 3x_0}\A_0(x_0,L)$, and the expression of the latter follows from \eqref{volumeofcriticalpointsinstatement}. The last claim is clear from the above.
\end{proof}
Thanks to \Cref{volbijlength}, we can re-state \Cref{uniquecriticalthm} as follows. \begin{corollary}\label{uniquecriticalthmconvolume}
Fix $x_0>0$ and $\A_0>0$. The following facts hold.
\begin{enumerate}
	\item $\F$ admits a vertically symmetric area-preserving critical point $\gamma$ with enclosed area $\A_0$ if and only if $\A_0>\frac{3}{2}\pi x_0^2$. In these cases, $\gamma=\gamma(x_0,L(x_0,\A_0))$ and it is unique.
	\item $\F-\lambda\A$ admits a vertically symmetric  critical point $\gamma$with enclosed area $\A$ if and only if $\A_0>\frac{3}{2}\pi x_0^2$. In these cases, $\gamma=\gamma(x_0,L(x_0,\A_0))$ and it is unique.
\end{enumerate}
\end{corollary}
\begin{remark}
Fix $x_0>0$. In view of \Cref{volbijlength}, in the following we may equivalently fix $L>3x_0$ or $\A_0>\frac{3}{2}\pi x_0^2$. In both cases, we tacitly mean that $\A_0=\A_0(x_0,L)$ and $L=L(x_0,\A_0)$ respectively.
\end{remark}
\begin{remark}\label{rem_funzionalisupunticritici}
Fix $0<3x_0<L$.
Recall that
\begin{equation*}
	H(s)=\frac{\pi}{\sigma\sqrt{\k^2-(s-L)^2}}.
\end{equation*}
Therefore
\begin{equation}\label{computeh}
	\begin{split}
		\F(\gamma(x_0,L))=\frac{ \k\sigma}{\pi}\int_{0}^{2L}\sqrt{1-\left(\frac{s-L}{\k}\right)^2}\,ds\overset{\eqref{changevar}}{=} \frac{\k^2\sigma}{\pi}\int_{-\sigma}^{\sigma}\cos^2 l\,dl 
		=\frac{\k^2\sigma}{2\pi}\left(2\sigma+\sin 2\sigma\right).
	\end{split}
\end{equation}
In particular,
\begin{equation*}
	\begin{split}
		\A(\gamma(x_0,L))&\overset{\eqref{computvolume}}{=}
		\frac{ \pi \k^2\sigma^2}{2(\pi^2-\sigma^2)}+\frac{\pi \k^2\sigma(\pi^2+\sigma^2)}{4(\pi^2-\sigma^2)^2}\sin2\sigma\\
		& \overset{\eqref{computeh}}{=}\frac{1}{\lambda}\left(\frac{2(\pi^2-\sigma^2)\sigma+(\pi^2+\sigma^2)\sin 2\sigma}{2(\pi^2-\sigma^2)\sigma+(\pi^2-\sigma^2)\sin 2\sigma}\right)\F(\gamma(x_0,L)).
	\end{split}
\end{equation*}
\end{remark}
\section{Second variation and stability}\label{sec_sectionstab}
In this section we introduce various notions of stability for the functionals $\F$ and $\F-\lambda\A$.  
\subsection{Second variation formulas}\label{sec_secvarform}
We begin computing the second variation formula for $\F$ when evaluated at critical points.
\begin{proposition}\label{propsecvaroff}
Fix $0<3x_0<L$. Let $\gamma=\gamma(x_0,L)$ be the unique critical point of \Cref{uniquecriticalthm}. Let $\lambda=\lambda(x_0,L)$ be as in \eqref{definitionoflambda}. Let $\mathcal V$ be a one-sided admissible variation of $\gamma$. Let $\varphi,\varphi_\tau,\psi,\psi_\tau\in C^\infty[0,2L]$ be such that 
\begin{equation}\label{quattrfunzioni}
	X(s)=\varphi(s)N(s)+\varphi_\tau(s)\dot\gamma(s)\qquad\text{and}\qquad X'(s)=\psi(s)N(s)+\psi_\tau(s)\dot\gamma(s)\qquad\text{ for any $s\in[0,2L]$. }
\end{equation}
Then
\begin{equation}\label{eq:sec}
	\begin{split}
		\left.\frac{d^2\F(\mathcal V(\cdot,t))}{dt^2}\right|_{t=0^+}&=\int_0^{2L}\left(2H\varphi^2-\frac{2\dot\varphi^2}{H}+\frac{2\ddot\varphi^2}{H^3}\right)\,ds\\
		&\quad+\lambda\int_0^{2L}\left(\psi+H\varphi_\tau^2+2\varphi\dot\varphi_\tau\right)\,ds+\left[\frac{\dot\psi-2\dot\varphi\dot\varphi_\tau}{H^2}\right]_{0}^{2L}.
	\end{split}
\end{equation}
If in addition $\mathcal V$ is an admissible variation, then $\dot\varphi(0)=\dot\varphi(2L)=0$, $\dot\psi(0)\leq 0$ and $\dot\psi(2L)\geq 0$, whence 
\begin{equation}\label{secvarFbetter}
	\begin{split}
		\left.\frac{d^2\F(\mathcal V(\cdot,t))}{dt^2}\right|_{t=0}&=\int_0^{2L}\left(2H\varphi^2-\frac{2\dot\varphi^2}{H}+\frac{2\ddot\varphi^2}{H^3}\right)\,ds+\lambda\int_0^{2L}\left(\psi+H\varphi_\tau^2+2\varphi\dot\varphi_\tau\right)\,ds+\left[\frac{\dot\psi}{H^2}\right]_{0}^{2L}\\
		&\geq \int_0^{2L}\left(2H\varphi^2-\frac{2\dot\varphi^2}{H}+\frac{2\ddot\varphi^2}{H^3}\right)\,ds+\lambda\int_0^{2L}\left(\psi+H\varphi_\tau^2+2\varphi\dot\varphi_\tau\right)\,ds.
	\end{split}
\end{equation}
\end{proposition}
\begin{proof}
As usual, we set $\gamma_t=\mathcal V(\cdot,t)$ for any $t\in[0,t_{\mathcal V}]$. As $\gamma_t(0)=(x_0,0)$ and $\gamma_t(2L)=(-x_0,0)$,
\begin{equation}\label{zeroorderbcsecondvar}
	\varphi(0)=\varphi_\tau(0)=\psi(0)=\psi_\tau(0)=\varphi(2L)=\varphi_\tau(2L)=\psi(2L)=\psi_\tau(2L)=0.
\end{equation}
Define $\alpha(s,t)=\langle\dot\gamma_t(s),\dot\gamma_t(s)\rangle$ and $\beta(s,t)=\langle\dot\gamma_t(s),\ddot\gamma_t(s)^\perp\rangle$ for any $t\in[0,\tv]$. For the sake of notational simplicity, we set $\dot\alpha_0=\left.\frac{\partial \alpha(\cdot,t)}{\partial t}\right|_{t=0^+}$ and $\ddot\alpha_0=\left.\frac{\partial^2 \alpha(\cdot,t)}{\partial t^2}\right|_{t=0^+}$, and accordingly for $\beta.$
Notice that $\alpha(s,0)\equiv 1$ and $\beta (s,0)=H(s)$ for any $s\in[0,2L]$. Then
\begin{multline}\label{eq:sec1}
	\left.\frac{\partial^2\left(\alpha^2\beta^{-1}\right)}{\partial t^2}\right|_{t=0^+}=\left.\frac{\partial}{\partial t}\left(2\alpha\frac{\partial\alpha}{\partial t}\beta^{-1}-\alpha^2\frac{\partial \beta}{\partial t}\beta^{-2}\right)\right|_{t=0^+}\\
	\begin{split}
		&=\left.\left(2\left(\frac{\partial \alpha}{\partial t}\right)^2\beta^{-1}+2\alpha\frac{\partial^2 \alpha}{\partial t^2}\beta^{-1}-4\alpha\frac{\partial \alpha}{\partial t}\frac{\partial \beta}{\partial t}\beta^{-2}-\alpha^2\frac{\partial^2\beta}{\partial t^2}\beta^{-2}+2\alpha^2\left(\frac{\partial \beta}{\partial t}\right)^2\beta^{-3}\right)\right|_{t=0^+}\\
		&=\frac{2\dot\alpha_0^2}{H}+\frac{2\ddot\alpha_0}{H}-\frac{4\dot\alpha_0\dot\beta_0}{H^2}-\frac{\ddot\beta_0}{H^2}+\frac{2\dot\beta_0^2}{H^3}\\
		&=\frac{2\ddot\alpha_0}{H}-\frac{\ddot\beta_0}{H^2}+\frac{2}{H}\left(\dot\alpha_0-\frac{\dot\beta_0}{H}\right)^2,
	\end{split}
\end{multline}
so that
\begin{equation}\label{firststepsecondvar}
	\left.\frac{d^2\F(\mathcal V(\cdot,t))}{dt^2}\right|_{t=0^+}=\int_{0}^{2L}\left.\frac{\partial^2 }{\partial t^2}\left(\frac{\langle\dot\gamma_t,\dot\gamma_t\rangle^2}{\langle\dot\gamma_t,\ddot\gamma_t^\perp\rangle}\right)\right|_{t=0^+}\,ds\overset{\eqref{eq:sec1}}{=}\int_0^{2L}\left(\frac{2\ddot\alpha_0}{H}-\frac{\ddot\beta_0}{H^2}+\frac{2}{H}\left(\dot\alpha_0-\frac{\dot\beta_0}{H}\right)^2\right)\,ds.
\end{equation}
We compute the integrand appearing in \eqref{firststepsecondvar}. Notice that 
\begin{align*}
	\frac{\partial\alpha}{\partial t}&=2\left\langle\frac{\partial \dot\gamma_t}{\partial t},\dot\gamma_t\right\rangle,\\
	\frac{\partial^2\alpha}{\partial t^2}&=2\left\langle\frac{\partial^2\dot\gamma_t}{\partial t^2},\dot\gamma_t\right\rangle+2\left\langle\frac{\partial \dot\gamma_t}{\partial t},\frac{\partial \dot\gamma_t}{\partial t}\right\rangle,\\
	\frac{\partial\beta}{\partial t}&=\left\langle\frac{\partial \dot\gamma_t}{\partial t},\ddot\gamma_t^\perp\right\rangle+\left\langle\dot\gamma_t,\frac{\partial\ddot\gamma_t^\perp}{\partial t}\right\rangle,\\
	\frac{\partial^2\beta}{\partial t^2}&=\left\langle\frac{\partial^2 \dot\gamma_t}{\partial t^2},\ddot\gamma_t^\perp\right\rangle+2\left\langle\frac{\partial \dot\gamma_t}{\partial t},\frac{\partial\ddot\gamma_t^\perp}{\partial t}\right\rangle+\left\langle\dot\gamma_t,\frac{\partial^2\ddot\gamma_t^\perp}{\partial t^2}\right\rangle.
\end{align*}
Therefore, recalling that $N=\dot\gamma^\perp$ and $\ddot\gamma^\perp=\dot N=H\dot\gamma$ (cf. \eqref{dotgammaeddotgamma}),
\begin{equation}\label{eq:sec2}
	\begin{split}
		\dot\alpha_0&=2\big\langle\dot X,\dot\gamma\big\rangle,\\
		\ddot\alpha_0&=2\big\langle\dot X',\dot\gamma\big\rangle+2\big\langle\dot X,\dot X\big\rangle,\\
		\dot \beta_0&=H\big\langle\dot X,\dot\gamma\big\rangle-\big\langle\ddot X,N\big\rangle,\\
		\ddot\beta_0&=H\big\langle\dot X',\dot\gamma\big\rangle+2\big\langle\dot X,\ddot X^\perp\big\rangle-\big\langle\ddot X',N\big\rangle.
	\end{split}
\end{equation}
By \eqref{quattrfunzioni},
\begin{equation}\label{dotxexprinproof}
	\dot X=\varphi\dot N+\dot\varphi N+\varphi_\tau\ddot\gamma+\dot\varphi_\tau\dot\gamma=H\varphi\dot\gamma+\dot\varphi N-H\varphi_\tau N+\dot\varphi_\tau\dot\gamma=\left(\dot\varphi-H\varphi_\tau\right)N+\left(\dot\varphi_\tau+H\varphi\right)\dot\gamma
\end{equation}
and
\begin{equation}\label{eq:sec3}
	\begin{split}
		\ddot X&=\left(\dot\varphi-H\varphi_\tau\right)\dot N+\left(\ddot\varphi-\dot H\varphi_\tau-H\dot \varphi_\tau\right)N+\left(\dot\varphi_\tau+H\varphi\right)\ddot\gamma+\left(\ddot\varphi_\tau+\dot H\varphi+H\dot\varphi\right)\dot\gamma\\
		&=\left(H\dot\varphi-H^2\varphi_\tau\right)\dot \gamma+\left(\ddot\varphi-\dot H\varphi_\tau-H\dot \varphi_\tau\right)N-\left(H\dot\varphi_\tau+H^2\varphi\right)N+\left(\ddot\varphi_\tau+\dot H\varphi+H\dot\varphi\right)\dot\gamma\\
		&=\left(\ddot\varphi-2H\dot \varphi_\tau-\dot H\varphi_\tau-H^2\varphi\right)N+\left(\ddot\varphi_\tau+2H\dot\varphi+\dot H\varphi-H^2\varphi_\tau\right)\dot\gamma.
	\end{split}
\end{equation}
Similarly, we have 
\begin{equation}\label{eq:sec4}
	\begin{split}
		\dot X'&=\left(\dot\psi-H\psi_\tau\right)N+\left(\dot\psi_\tau+H\psi\right)\dot\gamma,\\
		\ddot X'&=\left(\ddot\psi-2H\dot \psi_\tau-\dot H\psi_\tau-H^2\psi\right)N+\left(\ddot\psi_\tau+2H\dot\psi+\dot H\psi-H^2\psi_\tau\right)\dot\gamma.
	\end{split}
\end{equation}
Inserting \eqref{dotxexprinproof}, \eqref{eq:sec3} and \eqref{eq:sec4} in \eqref{eq:sec2}, we get
\begin{align*}
	\dot\alpha_0&=2\dot\varphi_\tau+2H\varphi,\\
	\ddot\alpha_0&=2\dot\psi_\tau+2H\psi+2\left(\dot\varphi-H\varphi_\tau\right)^2+2\left(\dot\varphi_\tau+H\varphi\right)^2\\
	&=2\dot\psi_\tau+2H\psi+2\dot\varphi^2-4H\dot\varphi\varphi_\tau+2H^2\varphi_\tau^2+2\dot\varphi_\tau^2+4H\varphi\dot\varphi_\tau+2H^2\varphi^2,\\
	\dot\beta_0&=H\dot\varphi_\tau+H^2\varphi-\ddot\varphi+2H\dot\varphi_\tau+\dot H\varphi_\tau+H^2\varphi\\
	&=2H^2\varphi+3H\dot\varphi_\tau+\dot H\varphi_\tau-\ddot\varphi,\\
	\ddot\beta_0&=H\dot\psi_\tau+H^2\psi+2\left(\dot\varphi-H\varphi_\tau\right)\left(\ddot\varphi_\tau+2H\dot\varphi+\dot H\varphi-H^2\varphi_\tau\right)\\
	&\quad-2\left(\dot\varphi_\tau+H\varphi\right)\left(\ddot\varphi-2H\dot \varphi_\tau-\dot H\varphi_\tau-H^2\varphi\right)-\left(\ddot\psi-2H\dot \psi_\tau-\dot H\psi_\tau-H^2\psi\right)\\
	&=2H^2\psi+3H\dot\psi_\tau+\dot H\psi_\tau-\ddot\psi\\
	&\quad+2\dot\varphi\ddot\varphi_\tau+4H\dot\varphi^2+2\dot H\varphi\dot\varphi-6H^2\dot\varphi\varphi_\tau-2H\varphi_\tau\ddot\varphi_\tau+2H^3\varphi_\tau^2\\
	&\quad-2\ddot\varphi\dot\varphi_\tau+4H\dot\varphi_\tau^2+2\dot H\varphi_\tau\dot\varphi_\tau+6H^2\varphi\dot\varphi_\tau-2H\varphi\ddot\varphi+2H^3\varphi^2.
\end{align*}
In particular,
\begin{equation}\label{eq:sec5}       
	\begin{split}
		\frac{\partial^2\left(\alpha^2\beta^{-1}\right)}{\partial t^2}\Big|_{t=0}&=\frac{2\ddot\alpha_0}{H}-\frac{\ddot\beta_0}{H^2}+\frac{2}{H}\left(\dot\alpha_0-\frac{\dot\beta_0}{H}\right)^2\\
		&=\frac{4\dot\psi_\tau}{H}+4\psi+\frac{4\dot\varphi^2}{H}-8\dot\varphi\varphi_\tau+4H\varphi_\tau^2+\frac{4\dot\varphi_\tau^2}{H}+8\varphi\dot\varphi_\tau+4H\varphi^2-2\psi\\
		&\quad -\frac{3\dot\psi_\tau}{H}-\frac{\dot H\psi_\tau}{H^2}+\frac{\ddot\psi}{H^2}-\frac{2\dot\varphi\ddot\varphi_\tau}{H^2}-\frac{4\dot\varphi^2}{H}-\frac{2\dot H\varphi\dot\varphi}{H^2}+6\dot\varphi\varphi_\tau+\frac{2\varphi_\tau\ddot\varphi_\tau}{H}-2H\varphi_\tau^2\\
		&\quad+\frac{2\ddot\varphi\dot\varphi_\tau}{H^2}-\frac{4\dot\varphi_\tau^2}{H}-\frac{2\dot H\varphi_\tau\dot\varphi_\tau}{H^2}-6\varphi\dot\varphi_\tau+\frac{2\varphi\ddot\varphi}{H}-2H\varphi^2+\frac{2}{H}\left(\dot\varphi_\tau+\frac{\dot H\varphi_\tau}{H}-\frac{\ddot\varphi}{H}\right)^2\\
		&=2\psi+\frac{\ddot\psi}{H^2}+\frac{\dot\psi_\tau}{H}+\frac{d(H^{-1})}{ds}\psi_\tau+2H\varphi^2+2\frac{d(H^{-1})}{ds}\varphi\dot\varphi+\frac{2\varphi\ddot\varphi}{H}+\frac{2\ddot\varphi^2}{H^3}\\
		&\quad+2H\varphi_\tau^2+\frac{2\dot\varphi_\tau^2}{H}+\frac{2\varphi_\tau\ddot\varphi_\tau}{H}-2\frac{d(H^{-1})}{ds}\varphi_\tau\dot\varphi_\tau-\dot H\frac{d(H^{-2})}{ds}\varphi_\tau^2\\
		&\quad-2\dot\varphi\varphi_\tau+2\varphi\dot\varphi_\tau-\frac{2\dot\varphi\ddot\varphi_\tau}{H^2}-\frac{2\ddot\varphi\dot\varphi_\tau}{H^2}+2\frac{d(H^{-2})}{ds}\ddot\varphi\varphi_\tau.
	\end{split}
\end{equation}
Now we compute the integral of the right-hand side of \eqref{eq:sec5}. Regarding the terms involving $\psi$ and $\psi_\tau$,
\begin{equation}\label{eq:term1}
	\begin{split}
		\int_0^{2L}\left(2\psi+\frac{\ddot\psi}{H^2}\right)\,ds&=\int_0^{2L}\left(2\psi-\frac{d(H^{-2})}{ds}\dot\psi\right)\,ds+\left[\frac{\dot\psi}{H^2}\right]_0^{2L}\\
		&\overset{\eqref{zeroorderbcsecondvar}}{=}\int_0^{2L}\left(2+\frac{d^2(H^{-2})}{ds^2}\right)\psi\,ds+\left[\frac{\dot\psi}{H^2}\right]_0^{2L}\\
		&\overset{\eqref{equazionepuntocriticoinlemma}}{=}\lambda\int_0^{2L}\psi\,ds+\left[\frac{\dot\psi}{H^2}\right]_0^{2L},
	\end{split}
\end{equation}
and moreover,
\begin{equation}\label{eq:term2}
	\int_{0}^{2L}\left(\frac{\dot\psi_\tau}{H}+\frac{d(H^{-1})}{ds}\psi_\tau\right)\,ds=\int_0^{2L}\frac{d}{ds}\left(\frac{\psi_\tau}{H}\right)\,ds\overset{\eqref{zeroorderbcsecondvar}}{=}0.
\end{equation}
For what concerns $\varphi$ and $\varphi_\tau$, notice first that
\begin{equation}\label{eq:term3}
	\int_{0}^{2L}\left(2H\varphi^2+2\frac{d(H^{-1})}{ds}\varphi\dot\varphi+\frac{2\varphi\ddot\varphi}{H}+\frac{2\ddot\varphi^2}{H^3}\right)\,ds\overset{\eqref{zeroorderbcsecondvar}}{=}\int_0^{2L}\left(2H\varphi^2-\frac{2\dot\varphi^2}{H}+\frac{2\ddot\varphi^2}{H^3}\right)\,ds.
\end{equation}
Moreover, 
\begin{equation}\label{eq:term4}
	\begin{split}
		\int_{0}^{2L}\Big(2H\varphi_\tau^2+&\frac{2\dot\varphi_\tau^2}{H}+\frac{2\varphi_\tau\ddot\varphi_\tau}{H}-2\frac{d(H^{-1})}{ds}\varphi_\tau\dot\varphi_\tau-\dot H\frac{d(H^{-2})}{ds}\varphi_\tau^2\Big)\,ds\\
		&\overset{\eqref{zeroorderbcsecondvar}}{=}\int_{0}^{2L}\left(2H\varphi_\tau^2+\frac{4\dot\varphi_\tau^2}{H}+\frac{4\varphi_\tau\ddot\varphi_\tau}{H}-\dot H\frac{d(H^{-2})}{ds}\varphi_\tau^2\right)\,ds\\
		&\overset{\eqref{zeroorderbcsecondvar}}{=}\int_{0}^{2L}\left(\left(2+\frac{d^2(H^{-2})}{ds^2}\right)H\varphi_\tau^2+\frac{4}{H}\frac{d(\varphi_\tau\dot\varphi_\tau)}{ds}+ 2H\frac{d(H^{-2})}{ds}\varphi_\tau\dot\varphi_\tau\right)\,ds\\
		&\overset{\eqref{equazionepuntocriticoinlemma}}{=}\lambda\int_{0}^{2L}H\varphi_\tau^2\,ds+\int_{0}^{2L}\left(\frac{4}{H}\frac{d(\varphi_\tau\dot\varphi_\tau)}{ds}+ 4\frac{d(H^{-1})}{ds}\varphi_\tau\dot\varphi_\tau\right)\,ds\\
		&=\lambda\int_{0}^{2L}H\varphi_\tau^2\,ds+\int_{0}^{2L}\frac{d}{ds}\left(\frac{4\varphi_\tau\dot\varphi_\tau}{H}\right)\,ds\\
		&\overset{\eqref{zeroorderbcsecondvar}}{=}\lambda\int_{0}^{2L}H\varphi_\tau^2\,ds.
	\end{split}
\end{equation}
Finally,
\begin{equation}\label{eq:term5}
	\begin{split}
		\int_0^{2L}\Big(-2\dot\varphi\varphi_\tau&+2\varphi\dot\varphi_\tau-\frac{2\dot\varphi\ddot\varphi_\tau}{H^2}-\frac{2\ddot\varphi\dot\varphi_\tau}{H^2}+2\frac{d(H^{-2})}{ds}\ddot\varphi\varphi_\tau\Big)\,ds\\
		&\overset{\eqref{zeroorderbcsecondvar}}{=}\int_0^{2L}\left(-4\dot\varphi\varphi_\tau+2\frac{d(H^{-2})}{ds}\dot\varphi\dot\varphi_\tau+2\frac{d(H^{-2})}{ds}\ddot\varphi\varphi_\tau\right)\,ds-\left[\frac{2\dot\varphi\dot\varphi_\tau}{H^2}\right]_{0}^{2L}\\
		&=\int_0^{2L}\left(-4\dot\varphi\varphi_\tau+2\frac{d(H^{-2})}{ds}\frac{d(\dot\varphi\varphi_\tau)}{ds}\right)\,ds-\left[\frac{2\dot\varphi\dot\varphi_\tau}{H^2}\right]_{0}^{2L}\\
		&\overset{\eqref{zeroorderbcsecondvar}}{=}-\int_0^{2L}2\dot\varphi\varphi_\tau\left(2+\frac{d^2(H^{-2})}{ds^2}\right)\,ds-\left[\frac{2\dot\varphi\dot\varphi_\tau}{H^2}\right]_{0}^{2L}\\
		&\overset{\eqref{zeroorderbcsecondvar},\eqref{equazionepuntocriticoinlemma}}{=}\lambda\int_0^{2L}2\varphi\dot\varphi_\tau\,ds-\left[\frac{2\dot\varphi\dot\varphi_\tau}{H^2}\right]_{0}^{2L}.
	\end{split}
\end{equation}
Substituting \eqref{eq:term1}, \eqref{eq:term2}, \eqref{eq:term3}, \eqref{eq:term4}, \eqref{eq:term5}  in \eqref{firststepsecondvar}, we get \eqref{eq:sec}. To conclude, assume that $\mathcal V$ is admissible. We know from the proof of \Cref{uniquecriticalthm} that $0$ is, respectively, minimum and maximum point of $\dot y_t(0)$ and of $\dot y_t(2L)$. 
Therefore
\begin{equation*}
	\left \langle \dot X(0),\frac{\partial}{\partial y}\right\rangle,\,  \left \langle \dot X(2L),\frac{\partial}{\partial y}\right\rangle=0,\qquad   \left \langle \dot X'(0),\frac{\partial}{\partial y}\right\rangle\geq 0\qquad\text{and}\qquad   \left \langle \dot X'(0),\frac{\partial}{\partial y}\right\rangle\leq 0.
\end{equation*}
Recalling that $\dot \gamma(0)=\dot\gamma(2L)=(1,0)$ and $N(0)=N(2L)=(0,-1)$, and by \eqref{zeroorderbcsecondvar}, we get $\dot\varphi(0)=\dot\varphi(2L)=0$, $\dot\psi(0)\leq 0$ and $\dot\psi(2L)\geq 0$, and the thesis follows.
\end{proof}
The second variation formula for the area is well-known. We prove it for the sake of completeness.
\begin{proposition}\label{secvarvolthm}
Fix $x_0>0$ and $L>0$. Let $\gamma:[0,2L]\to\rr^2$ be an admissible curve parametrized by arc-length. Let $\mathcal V$ be a one-sided admissible variation of $\gamma$. Let $\varphi,\varphi_\tau,\psi\in C^\infty[0,2L]$ be as in \eqref{quattrfunzioni}.
Then
\begin{equation}\label{secondvarvol}
	\left.\frac{d^2\A(\mathcal V(\cdot,t))}{dt^2}\right|_{t=0^+}=\int_0^{2L}H\varphi^2\,ds+\int_0^{2L}\left(\psi+H\varphi_\tau^2+2\varphi\dot\varphi_\tau\right)\,ds.
\end{equation}
\end{proposition}
\begin{proof}
Notice that
\begin{equation*}
	\begin{split}
		\left.\frac{\partial^2}{\partial t^2}\left(\frac{1}{2}\big\langle \gamma_t,\dot\gamma_t^\perp\big\rangle\right)\right|_{t=0^+}&=\left.\frac{\partial}{\partial t}\left(\frac{1}{2}\left\langle \frac{\partial\gamma_t}{\partial t},\dot\gamma_t^\perp\right\rangle+\frac{1}{2}\left\langle \gamma_t,\frac{\partial\dot\gamma_t^\perp}{\partial t}\right\rangle\right)\right|_{t=0^+}\\
		&=\left.\left(\frac{1}{2}\left\langle \frac{\partial^2\gamma_t}{\partial t^2},\dot\gamma_t^\perp\right\rangle+\left\langle \frac{\partial\gamma_t}{\partial t},\frac{\partial\dot\gamma_t^\perp}{\partial t}\right\rangle+\frac{1}{2}\left\langle \gamma_t,\frac{\partial^2\dot\gamma_t^\perp}{\partial t^2}\right\rangle\right)\right|_{t=0^+}\\
		&=\frac{1}{2}\left\langle X',N \right\rangle+\left\langle X,\dot X^\perp\right\rangle+\frac{1}{2}\left\langle \gamma,\dot X'^\perp\right\rangle.
	\end{split}
\end{equation*}
Moreover,
\begin{equation*}
	\frac{1}{2}\int_0^{2L}\left\langle \gamma,\dot X'^\perp\right\rangle\,ds\overset{\eqref{zeroorderbcsecondvar}}{=}-\frac{1}{2}\int_{0}^{2L}\left\langle \dot\gamma, X'^\perp\right\rangle\,ds=\frac{1}{2}\int_{0}^{2L}\left\langle X',N\right\rangle\,ds.
\end{equation*}
Recalling \eqref{dotxexprinproof},
\begin{align*}
	\left \langle X',N\right\rangle&=\psi,\\
	\left\langle X,\dot X^\perp\right\rangle&=H\varphi^2+H\varphi_\tau^2+\varphi\dot\varphi_\tau-\dot\varphi\varphi_\tau,
\end{align*}
whence, from \eqref{volumediomegaincurva} we get
\begin{equation*}
	\begin{split}
		\left.\frac{d^2\A(\mathcal V(\cdot,t))}{dt^2}\right|_{t=0^+}&=\int_{0}^{2L}\left.\frac{\partial^2}{\partial t^2}\left(\frac{1}{2}\big\langle \gamma_t,\dot\gamma_t^\perp\big\rangle\right)\right|_{t=0^+}\,ds\\
		&=\int_{0}^{2L}\left(\left\langle X',N \right\rangle+\left\langle X,\dot X^\perp\right\rangle\right)\,ds\\
		&\overset{\eqref{zeroorderbcsecondvar}}{=}\int_0^{2L}H\varphi^2\,ds+\int_0^{2L}\left(\psi+H\varphi_\tau^2+2\varphi\dot\varphi_\tau\right)\,ds.\qedhere
	\end{split}
\end{equation*}
\end{proof}
\begin{remark}
\Cref{secvarvolthm} applies not only to critical points, but to any admissible curve.
\end{remark}
Combining \Cref{propsecvaroff} and \Cref{secvarvolthm}, we obtain the second variation of $\F-\lambda\A$ when evaluated at critical points.
\begin{corollary}\label{corollstbofg}
Fix $0<3x_0<L$. Let $\gamma=\gamma(x_0,L)$ be the unique critical point of \Cref{uniquecriticalthm}. Let $\lambda=\lambda(x_0,L)$ be as in \eqref{definitionoflambda}. Let $\mathcal V$ be an admissible variation of $\gamma$. Let $\varphi,\psi\in C^\infty[0,2L]$ be as in \eqref{quattrfunzioni}. Then 
\begin{equation} \label{secvarG}
	\begin{split}
		\left.\frac{d^2\left(\F-\lambda\A\right)(\mathcal V(\cdot,t))}{dt^2}\right|_{t=0}&=\int_0^{2L}\left((2-\lambda)H\varphi^2-\frac{2\dot\varphi^2}{H}+\frac{2\ddot\varphi^2}{H^3}\right)\,ds+\left[\frac{\dot\psi}{H^2}\right]_{0}^{2L}\\
		&\geq \int_0^{2L}\left((2-\lambda)H\varphi^2-\frac{2\dot\varphi^2}{H}+\frac{2\ddot\varphi^2}{H^3}\right)\,ds.
	\end{split}
\end{equation}
\end{corollary}
\begin{remark}
Differently from \eqref{secvarFbetter}, the integral part of the second variation of $\F-\lambda \A$ depends uniquely on $\varphi$, the normal component of the velocity vector field of the variation. A similar phenomenon occurs when dealing with the area functional (cf. \cite{MR731682}).
\end{remark}
\subsection{Different notions of stability}\label{sec_difnotstab}
Owing to the second variation formulas of the previous section, we formulate the following notions of stability.
\begin{definition}\label{stabilitydefinition}
Fix $0<3x_0<L$. Let $\gamma=\gamma(x_0,L)$ be the unique critical point of \Cref{uniquecriticalthm}.

\begin{enumerate}
	\item [(i)] We say that $\gamma$ is an \emph{area-preserving stable critical point} of $\F$ if
	\begin{equation}\label{stabdef1}
		\left.\frac{d^2\F(\mathcal V(\cdot,t))}{dt^2}\right|_{t=0}\geq 0\quad\text{for any area-preserving admissible variation $\mathcal V$ of $\gamma$.}
	\end{equation}
	
	\item [(ii)]We say that $\gamma$ is a \emph{stable critical point} of $\F-\lambda\A$ if
	\begin{equation}\label{stabdef2}
		\left.\frac{d^2(\F-\lambda\A)(\mathcal V(\cdot,t))}{dt^2}\right|_{t=0}\geq 0\quad\text{for any admissible variation $\mathcal V$ of $\gamma$.}
	\end{equation}
	
	\item [(iii)] We say that $\gamma$ is a \emph{first-order area-preserving stable critical point} of $\F-\lambda\A$ if
	\begin{equation}\label{stabdef3}
		\left.\frac{d^2(\F-\lambda\A)(\mathcal V(\cdot,t))}{dt^2}\right|_{t=0}\geq 0\quad\text{for any first-order area-preserving admissible variation $\mathcal V$ of $\gamma$.}
	\end{equation}
	
\end{enumerate}
\end{definition}
Properties \eqref{stabdef1} and \eqref{stabdef2} of the above definition are, respectively, the natural notions of stability associated with $\F$ and $\F-\lambda\A$. On the other hand, property \eqref{stabdef2} is clearly stronger than property \eqref{stabdef3}, since stability is tested on a wider class of variations. With the next proposition, we show that properties \eqref{stabdef1} and \eqref{stabdef3} are actually equivalent.
\begin{proposition}\label{stabilityequivalent}
Fix $0<3x_0<L$. Let $\gamma=\gamma(x_0,L)$ be the unique critical point of \Cref{uniquecriticalthm}. The following are equivalent. 
\begin{enumerate}
	\item [(i)] $\gamma$ is an area-preserving stable critical point of $\F$.
	\item [(ii)] $\gamma$ is a first-order area-preserving stable critical point of $\F-\lambda\A$.
\end{enumerate}
\end{proposition}
\begin{proof}
First we prove $(ii)\implies (i)$. Let $\mathcal{V}$ be an area-preserving admissible variation of $\gamma$. Let the functions $\varphi,\varphi_\tau,\psi\in C^\infty[0,2L]$ be as in \eqref{quattrfunzioni}. As $\A(\mathcal V(\cdot,t))$ is constant in $t$, \eqref{firstvarvol} and \eqref{secondvarvol} yield
\begin{equation}\label{davolpresafirstordervolpres}
	\int_0^{2L}\varphi\,ds=0\qquad\text{and}\qquad \int_0^{2L}\left(\psi+H\varphi_\tau^2+2\varphi\dot\varphi_\tau\right)\,ds=-\int_0^{2L}H\varphi^2\,ds,
\end{equation}
whence
\begin{equation*}
	\begin{split}
		\left.\frac{d^2\F(\mathcal V(\cdot,t))}{dt^2}\right|_{t=0}&\overset{\eqref{secvarFbetter}}{=}\int_0^{2L}\left(2H\varphi^2-\frac{2\dot\varphi^2}{H}+\frac{2\ddot\varphi^2}{H^3}\right)\,ds+\lambda\int_0^{2L}\left(\psi+H\varphi_\tau^2+2\varphi\dot\varphi_\tau\right)\,ds+\left[\frac{\dot\psi}{H^2}\right]_{0}^{2L}\\
		&\overset{\eqref{davolpresafirstordervolpres}}{=}\int_0^{2L}\left((2-\lambda)H\varphi^2-\frac{2\dot\varphi^2}{H}+\frac{2\ddot\varphi^2}{H^3}\right)\,ds+\left[\frac{\dot\psi}{H^2}\right]_{0}^{2L}\\
		&\overset{\eqref{secvarG}}{=}\left.\frac{d^2(\F-\lambda\A)(\mathcal V(\cdot,t))}{dt^2}\right|_{t=0}\\
		&\geq 0.
	\end{split}
\end{equation*}
To prove $(i)\implies (ii)$, let $\mathcal{V}$ be a first-order area-preserving admissible variation of $\gamma$. Let $ \mathcal {\tilde V}$ be the area-preserving variation of \Cref{docarmo}. Then 
\begin{equation*}
	\left.\frac{d^2(\F-\lambda\A)(\mathcal V(\cdot,t))}{dt^2}\right|_{t=0}\overset{\eqref{docarmopropultimateequation},\eqref{secvarG}}{=} \left.\frac{d^2(\F-\lambda\A)(\mathcal {\tilde V}(\cdot,t))}{dt^2}\right|_{t=0}=\left.\frac{d^2F(\mathcal {\tilde V}(\cdot,t))}{dt^2}\right|_{t=0}\geq 0,
\end{equation*}
an the thesis follows.
\end{proof}

\section{Stability of critical points}\label{sec_stability}
In this section, we show that the critical points emerging from \Cref{uniquecriticalthm} are stable critical points of $\F-\lambda\A$, i.e. the strongest sense of stability proposed in \Cref{stabilitydefinition}. This fact constitutes a remarkable difference with the case of the area functional. Indeed, spherical caps in $\rr^{n}$ may be stable for volume-preserving variations of the area functional but unstable for general variations of the perturbed functional associated with the area functional (cf. \cite{MR731682}). Our approach consists in providing sharp lower bounds of the form
\begin{equation}\label{sharplowerbound}
\mu_{\mathcal W}\int_0^{2L}\frac{2\dot\varphi^2}{H}\,ds\leq \int_0^{2L}\left(\frac{2\ddot\varphi^2}{H^3}+(2-\lambda)H \varphi^2\right)\,ds,
\end{equation}
where $\varphi$ belongs to a chosen functional space $\mathcal W$. The nature of the information we can infer from \eqref{sharplowerbound} depends on the choice of $\mathcal W$ and on $\mu_\mathcal W$.
\begin{example}[Stability]\label{examplestability}
Let $\mathcal V$ be an admissible variation of $\gamma$ with velocity vector field $X=\varphi N+\varphi_\tau\dot\gamma$. We already know that $\varphi(0),\varphi(2L)=0$ and $\dot\varphi(0),\dot\varphi(2L)=0$. Therefore, setting 
\begin{equation*}
	\mathcal W_1=\left\{\varphi\in W^{2,2}(0,2L)\,:\,\varphi(0),\varphi(2L)=0,\,\dot\varphi(0),\dot\varphi(2L)=0\right\},
\end{equation*}
if \eqref{sharplowerbound} holds with $\mu_{\mathcal W_1}\geq 1$, then \Cref{corollstbofg} implies that $\gamma$ is a stable critical point of $\F-\lambda\A$.
\end{example}
\begin{example}[One-sided area-preserving stability]\label{onesidedexample}
Let $\mathcal V$ be a one-sided, first-order area-preserving admissible variation of $\gamma$. Let $\varphi$ be as in \eqref{quattrfunzioni}. Then $\varphi(0),\varphi(2L)=0$ and $\int_0^{2L}\varphi=0$. Therefore, setting 
\begin{equation*}
	\mathcal W_2=\left\{\varphi\in W^{2,2}(0,2L)\,:\,\varphi(0)=\varphi(2L)=0,\,\int_0^{2L}\varphi=0\right\},
\end{equation*}
if \eqref{sharplowerbound} holds with $\mu_{\mathcal W_2}\geq 1$, then \Cref{corollstbofg} and \Cref{stabilityequivalent} imply that  \begin{equation}\label{onesidedstabilitiquasidef}
	\left.\frac{d^2\F(\mathcal V(\cdot,t))}{dt^2}\right|_{t=0^+}\geq 0
\end{equation}
for any area-preserving one-sided admissible variation $\mathcal V$ of $\gamma$. Although improperly, we may refer to \eqref{onesidedstabilitiquasidef} as one-sided area-preserving stability.
\end{example}
Being interested in stability, we mainly focus on \Cref{examplestability}. Indeed, although possibly interesting, the one-sided stability proposed in \Cref{onesidedexample} is not necessary to minimality (roughly speaking, the second derivative of a function at a boundary minimum point may be either positive or negative). Throughout this section, we may adopt the notation
\begin{equation*}
r=\frac{2}{H^3},\qquad q=\frac{2}{H},\qquad p=(2-\lambda) H.
\end{equation*}
\subsection{Existence of minimizers for \eqref{sharplowerbound}} \label{sec_exminprob}
Fix $0<3x_0<L$. Let $\gamma=\gamma(x_0,L)$ be the unique critical point of \Cref{uniquecriticalthm}. Consider the space $\mathcal S=W^{2,2}(0,2L)\cap W^{1,2}_0(0,2L)$. We endow it with the norm 
\begin{equation*}
\|u\|^2_{\mathcal S}=\|u\|^2_{W^{2,2}(0,2L)}=\|u\|^2_{L^2(0,2L)}+\|\dot u\|^2_{L^2(0,2L)}+\|\ddot u\|^2_{L^2(0,2L)}.
\end{equation*}
Denote by $\mathcal W$ any closed subspace of $\mathcal S$. 
Set 
\begin{equation*}
a(u,v)=\int_0^{2L}\frac{2\dot u\dot v}{H}\,ds,\qquad    b(u,v)=\int_0^{2L}\left(\frac{2\ddot u\ddot v}{H^3}+(2-\lambda)H u v\right)\,ds\qquad\text{for any $u,v\in \mathcal W$.}
\end{equation*}
Consider 
\begin{equation}\label{minprobstab}
\mu_{\mathcal W}=\inf_{u\in\mathcal W,u\neq 0}\frac{b(u,u)}{a(u,u)}.
\end{equation}
First, we show that \eqref{minprobstab} attains its minimum.
\begin{proposition}\label{exminstab}
Let $\mathcal W$ be any closed subspace of $\mathcal S$.   Then, there exists $u\in\mathcal W\setminus\{0\}$ such that
\begin{equation*}
	\mu_{\mathcal W}=\frac{b(u,u)}{a(u,u)}.
\end{equation*}
In particular, $\mu_{\mathcal W}>0$.
\end{proposition}
\begin{proof}
Since $\mathcal S$ is a closed subspace of $W^{2,2}(0,2L)$, it is a reflexive Banach space. Since $\mathcal W$ is closed, it is a reflexive Banach space as well.  By definition, $\mu_{\mathcal W}\geq 0$. Let $(u_h)_h\subseteq \mathcal W$ be a minimizing sequence. Since $a$ and $b$ are $2$-homogeneous, we assume that $a(u_h,u_h)=1$ for any $h$. Up to a subsequence, 
\begin{equation*}
	b(u_h,u_h)\leq \mu_{\mathcal W}+1.
\end{equation*}
We claim that there exists $c_1>0$ such that
\begin{equation}\label{ineqforcoerc}
	c_1\|u\|^2_\mathcal S\leq \|u\|^2_{L^2(0,2L)}+\|\ddot u\|^2_{L^2(0,2L)}\qquad\text{  for any $u\in \mathcal S$.}
\end{equation}
Indeed, let $u\in S$. As $u(0)=u(2L)=0$, then $\int_0^{2L}\dot u\,ds=0$. Therefore, by Poincaré-Wirtinger's inequality (cf. \cite{MR2759829}), there exists $c_2>0$ such that
\begin{equation*}
	\|\dot u\|_{L^2(0,2L)}\leq c_2\|\ddot u\|_{L^2(0,2L)},
\end{equation*}
and \eqref{ineqforcoerc} follows.
Therefore, there exists $\tilde c>0$ such that
\begin{equation*}
	b(u_h,u_h)\geq \min\left\{\min_{[0,2L]}\frac{2}{H^3},\min_{[0,2L]}(2-\lambda)H\right\}\left(\|u_h\|^2_{L^2(0,2L)}+\|\ddot u_h\|^2_{L^2(0,2L)}\right)\overset{\eqref{ineqforcoerc}}{\geq} \tilde c\|u_h\|^2_\mathcal S
\end{equation*}
In particular, $(u_h)_h$ is bounded. By reflexivity, up to a subsequence, there exists $u\in \mathcal W$ such that:
\begin{enumerate}
	\item [(i)]$u_h\to u$ uniformly on $[0,2L]$;
	\item [(ii)]$\dot u_h\to \dot u$ uniformly on $[0,2L]$;
	\item [(iii)]$\ddot u_h\to \ddot u$ weakly in $L^2(0,2L)$.
\end{enumerate}
Since 
\begin{equation*}
	v\mapsto\int_0^{2L}(2-\lambda)H v^2\,ds,\qquad v\mapsto\int_0^{2L}\frac{2\dot v^2}{H}\,ds\qquad\text{and }\qquad v\mapsto\int_0^{2L}\frac{2\ddot v^2}{H^3}\,ds
\end{equation*}
are, respectively, continuous with respect to the strong convergence of $L^2(0,2L)$, continuous with respect to the strong convergence of $L^2(0,2L)$, and lower-semicontinuous with respect to the weak convergence of $L^2(0,2L)$, we conclude that $a(u,u)=1$, and
\begin{equation*}
	\mu_{\mathcal W}\leq b(u,u)\leq\liminf_{h\to\infty}b(u_h,u_h)=\mu_{\mathcal W}.\qedhere
\end{equation*}
\end{proof}
\subsection{Euler-Lagrange equations and regularity}\label{sec_elforstab}
Next, we characterize $\mu_\mathcal W$ in terms of the Euler-Lagrange equation associated with \eqref{sharplowerbound}.
Indeed, if $u\in\mathcal W$ is a minimizer of \eqref{minprobstab}, then $u$ minimizes the functional
\begin{equation*}
v\mapsto b(v,v)-\mu_{\mathcal W} a(v,v)= \int_0^{2L}\left(\frac{2\ddot v^2}{H^3}-\mu_{\mathcal W}\frac{2\dot v^2}{H}+(2-\lambda)H v^2\right)\,ds,\qquad v\in\mathcal W.
\end{equation*}
Therefore it satisfies the Euler-Lagrange equation
\begin{equation}\label{weakelstab}
\int_0^{2L}\left(r \ddot u\ddot v-\mu_{\mathcal W} q\dot u\dot v+pu v\right)\,ds=0,\qquad \text{for any }v\in \mathcal W.
\end{equation}
\begin{proposition}
Let $\mathcal W$ be a closed subspace of $\mathcal S$ satisfying 
\begin{equation}\label{contieneleamedianulla}
	\left\{\varphi\in C^\infty_c(0,2L)\,:\,\int_0^{2L}\varphi\,ds=0\right\}\subseteq\mathcal W.
\end{equation}
Assume that $u\in\mathcal W$ solves \eqref{weakelstab}. Then $u\in C^\infty[0,2L]$.
\end{proposition}
\begin{proof}
Set $f=-\mu_{\mathcal W}\dot q\dot u-\mu_{\mathcal W} q\ddot u-p u$. Then $f\in L^2(0,2L)$. By \eqref{contieneleamedianulla} and \eqref{weakelstab},
\begin{equation*}
	T(\varphi):=   \int_0^{2L}r\ddot u\ddot \varphi\,ds-\int_0^{2L}f\varphi\,ds=0\qquad\text{for any }\varphi\in C^\infty_c(0,2L)\text{ such that }\int_0^{2L}\varphi\,ds=0.
\end{equation*}
Let $\varphi\in C^\infty_c(0,2L)$. Let $\psi\in C^\infty_c(0,2L)$ be such that $\int_0^{2L}\psi\,ds\neq 0$, and set $$\varphi_0=\varphi-\frac{\int_0^{2L}\varphi\,ds}{\int_0^{2L}\psi\,ds}\psi$$
Then $\int_0^{2L}\varphi_0\,ds=0$, so that 
\begin{equation*}
	0=T(\varphi_0)=T(\varphi)-\frac{\int_0^{2L}\varphi\,ds}{\int_0^{2L}\psi\,ds} T(\psi)=\int_0^{2L}r\ddot u\ddot \varphi\,ds-\int_0^{2L}\left(f+\frac{T(\psi)}{\int_0^{2L}\psi\,d\sigma}\right)\varphi\,ds.
\end{equation*}
In particular, $r\ddot u\in W^{2,2}(0,2L)$, whence $u\in H^4(0,2L)$. Bootstrapping this argument, $u\in H^{2k}(0,2L)$ for any $k\in\mathbb N$. By \cite[Theorem 8.8]{MR2759829}, $u\in C^\infty[0,2L]$.
\end{proof}
Let $u \in C^\infty[0,2L]$ be any solution to \eqref{weakelstab}. Then,  we integrate by parts \eqref{weakelstab} to get
\begin{equation}\label{weakstabwithbc}
\int_0^{2L}\left(\left(r \ddot u\right)''+\mu_{\mathcal W} \left(q \dot u\right)'+pu\right)v\,ds+\left[r\ddot u\dot v\right]^{2L}_0=0\qquad \text{for any $v\in \mathcal W$}.
\end{equation}
In particular, since we are assuming \eqref{contieneleamedianulla}, there exists $\nu\in\rr$ such that 
\begin{equation}\label{primavoltaequazionenonhom}
\left(r \ddot u\right)''+\mu_{\mathcal W} \left(q \dot u\right)'+pu=\nu\qquad\text{on }(0,2L).
\end{equation}
\subsection{Solutions to \eqref{primavoltaequazionenonhom}}\label{sec_soluzioniiiiiiiiiii} In this section we find the general solution to \begin{equation}\label{secondtimenonhom}
\left(r \ddot u\right)''+\mu \left(q \dot u\right)'+pu=\nu\qquad\text{on }(0,2L),
\end{equation}
where $\mu>0$ and $\nu\in\rr$. We begin by finding the general solution to its homogeneous counterpart, namely 
\begin{equation}\label{onlyequation}
\left(r \ddot u\right)''+\mu \left(q \dot u\right)'+pu=0\qquad\text{on }(0,2L),
\end{equation}
\begin{proposition}\label{soluzgenrralestab}
Let $\mu_0=-\frac{x_0}{L+x_0}+2\sqrt{\frac{x_0}{L+x_0}}$. Then, the following facts hold.
\begin{enumerate}
	\item [(i)] Assume that $\mu<\mu_0$.
	Then any solution to \eqref{onlyequation} is a linear combination of the functions    \begin{equation}\label{soluzionicasocomplesso}
		u_1=\cos \alpha\theta\sinh\beta\theta,\qquad
		u_2=\sin \alpha\theta\cosh\beta\theta,\qquad
		u_3= \sin \alpha\theta\sinh\beta\theta,\qquad
		u_4=\cos \alpha\theta\cosh\beta\theta,
	\end{equation}
	where
	\begin{equation}\label{alfabetacomplesso}
		\alpha=\frac{1}{2}\sqrt{\mu+\frac{x_0}{L+x_0}+2\sqrt{\frac{x_0}{L+x_0}}},\qquad\beta=\frac{1}{2}\sqrt{-\mu-\frac{x_0}{L+x_0}+2\sqrt{\frac{x_0}{L+x_0}}}.
	\end{equation}
	\item [(ii)] Assume that $\mu=\mu_0$.
	Then any solution to \eqref{onlyequation} is a linear combination of the functions     \begin{equation}\label{soluzionicasocritico}
		u_1=\sin\alpha \theta,\qquad
		u_2=\cos \alpha\theta,\qquad
		u_3= \theta\sin \alpha\theta,\qquad
		u_4=\theta\cos \alpha\theta.
	\end{equation}
	\item [(iii)] Assume that $\mu>\mu_0$.
Then any solution to \eqref{onlyequation} is a linear combination of the functions
\begin{equation}\label{soluzionicasoreale}
	u_1=\sin\left(\alpha-\gamma\right) \theta,\qquad
	u_2=\sin \left(\alpha+\gamma\right)\theta,\qquad
	u_3= \cos (\alpha-\gamma)\theta,\qquad
	u_4=\cos (\alpha+\gamma)\theta,
\end{equation}
where
\begin{equation}\label{defigamma}
	\gamma=\frac{1}{2}\sqrt{\mu+\frac{x_0}{L+x_0}-2\sqrt{\frac{x_0}{L+x_0}}}.
\end{equation}
\end{enumerate}
\end{proposition}
\begin{proof}
Fix $z\in\mathbb C$, and set $u:[0,2L]\to\mathbb C$ by $u(s)=\sin z\theta(s)$. Then
\begin{equation*}
\begin{split}
	\left(r  u''\right)''+ \mu\left(q  u'\right)' + pu&= \left(2 H^{-3} (\sin z\theta)''\right)''+ \mu\left(2H^{-1} (\sin z\theta)'\right)' + (2-\lambda)H\sin z\theta\\
	&=\left(2z H^{-3} (H\cos z \theta)'\right)''+ \mu\left(2 z \cos z \theta\right)' + (2-\lambda)H\sin z \theta\\
	&=\left(- z \left( H^{-2}\right)' \cos z \theta-2  z ^2H^{-1}\sin z \theta\right)''+\left(-2\mu z^2+2-\lambda\right)H\sin z\theta\\
	&=\left(-z\left( H^{-2}\right)'' \cos z\theta-2z^3\cos\theta\right)'+\left(-2\mu z^2+2-\lambda\right)H\sin z\theta\\
	&\overset{\eqref{equazionepuntocriticoinlemma}}{=} \left(\left(z(2-\lambda)-2z^3\right)\cos z\theta\right)'+\left(-2\mu z^2+2-\lambda\right)H\sin z\theta\\
	&=\left(2z^4-(2\mu+2-\lambda)z^2+(2-\lambda)\right)H\sin z\theta,\\
\end{split}
\end{equation*}
where we used the fact that $\dot\theta=H$ and $H (H^{-2})'=2(H^{-1})'$. In the same way, $v:[0,2L]\to\mathbb C$ defined by $v(s)=\cos z\theta (s)$ satisfies
\begin{equation*}
\left(r  v''\right)''+ \mu\left(q  v'\right)' + pv=\left(2z^4-(2\mu+2-\lambda)z^2+(2-\lambda)\right)H\cos z\theta.
\end{equation*}
In particular, $u$ and $v$ solve \eqref{onlyequation} provided that 
\begin{equation}\label{complexpolinomio}
2z^4-(2\mu+2-\lambda)z^2+(2-\lambda)=0.
\end{equation}
In this case, also the real and imaginary part $u$ and $v$ solve \eqref{onlyequation}. When $\mu<\mu_0$, the four solutions to \eqref{complexpolinomio} are $\pm\alpha\pm\mathrm i \beta$ and $\pm\alpha\mp \mathrm i \beta.$ Hence, the functions $u_1,\ldots,u_4$ given in \eqref{soluzionicasocomplesso} are solutions to \eqref{onlyequation}. When instead $\mu>\mu_0$, the four solutions to \eqref{complexpolinomio} are $\pm\alpha\pm\gamma$ and $\pm\alpha\mp\gamma$, and the functions given in \eqref{soluzionicasoreale} are solutions to \eqref{onlyequation}.
Finally, when $\mu=\mu_0$, \eqref{complexpolinomio} admits two real solutions, $\pm\alpha$. In particular, the functions $u_1$ and $u_2$ given in \eqref{soluzionicasocritico} are solutions. We claim that $u_3$ and $u_4$ are solutions too. To this aim, since $\mu=\mu_0$,
\begin{equation}\label{auxcomputcritic}
4\alpha^2=4\sqrt{\frac{x_0}{L+x_0}}\qquad\text{and}\qquad 2-\lambda+2\mu=4\sqrt{\frac{x_0}{L+x_0}}.
\end{equation}
Therefore,
\begin{equation*}      \begin{split}          \left(r \dot u''\right)''+ \mu&\left(q \dot u'\right)' + pu=\left(2H^{-3}(\theta\sin\alpha\theta)''\right)''+2\mu\left(H^{-1}(\theta\sin\alpha\theta)'\right)'+(2-\lambda)H\theta\sin\alpha\theta\\
	&=\left(2H^{-3}(H\sin\alpha\theta+\alpha H\theta\cos\alpha\theta)'\right)''+2\mu\left(\sin\alpha\theta+\alpha\theta\cos\alpha\theta\right)'+(2-\lambda)H\theta\sin\alpha\theta\\
	&=\left(-\left(H^{-2}\right)'\sin\alpha\theta+4\alpha H^{-1}\cos\alpha\theta-\alpha \left(H^{-2}\right)'\theta\cos\alpha\theta-2\alpha^2 H^{-1}\theta\sin\alpha\theta\right)''\\
	&\quad+4\mu\alpha H\cos\alpha\theta+\left(-2\mu\alpha^2 +(2-\lambda)\right)H\theta\sin\alpha\theta\\
	&=\left(\left(2-\lambda-6\alpha^2\right)\sin\alpha\theta+\alpha\left(2-\lambda-2\alpha^2\right)\theta\cos\alpha\theta\right)'\\
	&\quad+4\mu\alpha H\cos\alpha\theta+\left(-2\mu\alpha^2 +(2-\lambda)\right)H\theta\sin\alpha\theta\\
	&=2\alpha\left(2-\lambda-4\alpha^2+2\mu\right)H\cos\alpha\theta+\left(2\alpha^4-\left(2\mu+2-\lambda\right)\alpha^2+(2-\lambda)\right)H\theta\sin\alpha\theta\\
	&\overset{\eqref{complexpolinomio}}{=}2\alpha\left(-4\alpha^2+2-\lambda+2\mu\right)H\cos\alpha\theta\\
	&\overset{\eqref{auxcomputcritic}}{=}0,
\end{split}
\end{equation*}
whence $u_3$ is a solution. Similarly,  $u_4$ is a solution too.     
We are left to prove that the above solutions are linearly independent. To this aim, let $A_{ij}=u^{(i-1)}_j(0)$ for $i,j=1,\ldots,4$. Then $u_1,\ldots,u_4$ are linearly independent if and only if $\det A\neq 0$. 
Assume first that $\mu<\mu_0$ and $u_1,\ldots,u_4$ are the functions given by \eqref{soluzionicasocomplesso}.
We recall that
\begin{equation*}
u_1=\cos \alpha\theta\sinh\beta\theta,\qquad
u_2=\sin \alpha\theta\cosh\beta\theta,\qquad
u_3= \sin \alpha\theta\sinh\beta\theta,\qquad
u_4=\cos \alpha\theta\cosh\beta\theta.
\end{equation*}
Then,
\begin{align*}
\dot u_1&=H\left(-\alpha\sin\alpha\theta\sinh\beta\theta+\beta\cos\alpha\theta\cosh\beta\theta\right),\qquad\dot u_2=H\left(\alpha\cos\alpha\theta\cosh\beta\theta+\beta\sin\alpha\theta\sinh\beta\theta\right),\\
\dot u_3&=H\left(\alpha\cos\alpha\theta\sinh\beta\theta+\beta\sin\alpha\theta\cosh\beta\theta\right),\qquad\dot u_4=H\left(-\alpha\sin\alpha\theta\cosh\beta\theta+\beta\cos\alpha\theta\sinh\beta\theta\right).
\end{align*}
Moreover,
\begin{align*}
\ddot u_1&=\dot H\left(-\alpha\sin\alpha\theta\sinh\beta\theta+\beta\cos\alpha\theta\cosh\beta\theta\right)+H^2\left(-2\alpha\beta\sin\alpha\theta\cosh\beta\theta+\left(\beta^2-\alpha^2\right)\cos\alpha\theta\sinh\beta\theta\right),\\
\ddot u_2&=\dot H\left(\alpha\cos\alpha\theta\cosh\beta\theta+\beta\sin\alpha\theta\sinh\beta\theta\right)+H^2\left(2\alpha\beta\cos\alpha\theta\sinh\beta\theta+\left(\beta^2-\alpha^2\right)\sin\alpha\theta\cosh\beta\theta\right),\\
\ddot u_3&=\dot H\left(\alpha\cos\alpha\theta\sinh\beta\theta+\beta\sin\alpha\theta\cosh\beta\theta\right)+H^2\left(2\alpha\beta\cos\alpha\theta\cosh\beta\theta+\left(\beta^2-\alpha^2\right)\sin\alpha\theta\sinh\beta\theta\right),\\
\ddot u_4&=\dot H\left(-\alpha\sin\alpha\theta\cosh\beta\theta+\beta\cos\alpha\theta\sinh\beta\theta\right)+H^2\left(-2\alpha\beta\sin\alpha\theta\sinh\beta\theta+\left(\beta^2-\alpha^2\right)\cos\alpha\theta\cosh\beta\theta\right).
\end{align*}
Finally, recalling that $\theta(0)=0,$
\begin{align*}
\dddot u_1(0)&=\left(\ddot H(0)+H(0)^3\left(\beta^2-3\alpha^2\right)\right)\beta,\\ \dddot u_2(0)&=\left(\ddot H(0)+H(0)^3\left(3\beta^2-\alpha^2\right)\right)\alpha,\\  \dddot u_3(0)&=6H(0)\dot H(0)\alpha\beta.
\end{align*}
Therefore, 
\begin{equation*}
A=\begin{bmatrix}
	0 & 0 & 0 & 1 \\
	H(0)\beta & H(0)\alpha & 0 & \dot u_4(0) \\
	\dot H(0)\beta & \dot H(0)\alpha & 2H(0)^2\alpha\beta & \ddot u_4(0) \\
	\left(\ddot H(0)+H(0)^3\left(\beta^2-3\alpha^2\right)\right)\beta & \left(\ddot H(0)+H(0)^3\left(3\beta^2-\alpha^2\right)\right)\alpha & 6H(0)\dot H(0)\alpha\beta & \dddot u_4(0)
\end{bmatrix}.
\end{equation*}
In particular,
\begin{equation*}
\begin{split}
	\det A&=H(0)\beta\left(6H(0)\dot H(0)^2\alpha^2\beta-2H(0)^2\left(\ddot H(0)+H(0)^3\left(3\beta^2-\alpha^2\right)\right)\alpha^2\beta\right)\\
	&\quad-H(0)\alpha\left(6H(0)\dot H(0)^2\alpha\beta^2-2H(0)^2\left(\ddot H(0)+H(0)^3\left(\beta^2-3\alpha^2\right)\right)\alpha\beta^2\right)\\
	&=-4H(0)^6\alpha^2\beta^2\left(\alpha^2+\beta^2\right).
\end{split}
\end{equation*}
Since $H(0),\alpha,\beta\neq 0$, we conclude that $\det A\neq 0$.
Assume instead that $\mu>\mu_0$ and $u_1,\ldots,u_4$ are the functions given by \eqref{soluzionicasoreale}. 
Set $z_1=\alpha-\gamma$ and $z_2=\alpha+\gamma$. 
Then $z_1,z_2\in\rr$, $z_1\neq z_2$ and
\begin{equation*}
u_1=\sin z_1 \theta,\qquad
u_2=\sin z_2\theta,\qquad
u_3= \cos z_1 \theta,\qquad
u_4=\cos z_2\theta.
\end{equation*}
Therefore
\begin{equation*}
\dot u_1=Hz_1\cos z_1 \theta,\qquad
\dot u_2=Hz_2\cos z_2\theta,\qquad
\dot u_3= -Hz_1\sin z_1 \theta,\qquad
\dot u_4=-Hz_2\sin z_2\theta.
\end{equation*}
Moreover,
\begin{align*}
\ddot u_1&=\dot Hz_1\cos z_1\theta-H^2z_1^2\sin z_1 \theta,\qquad \ddot u_2=\dot Hz_2\cos z_2\theta-H^2z_2^2\sin z_2 \theta,\\
\ddot u_3&=-\dot Hz_1\sin z_1\theta-H^2z_1^2\cos z_1 \theta,\qquad \ddot u_4=-\dot Hz_2\sin z_2\theta-H^2z_2^2\cos z_2 \theta.
\end{align*}
Finally,
\begin{align*}
\dddot u_1(0)&=\left(\ddot H(0)-H(0)^3z_1^2\right)z_1,\qquad  \dddot u_2(0)=\left(\ddot H(0)-H(0)^3z_2^2\right)z_2,\\
\dddot u_3(0)&=-3H(0)\dot H(0)z_1^2,\qquad  \dddot u_4(0)=-3H(0)\dot H(0)z_2^2.
\end{align*}
Therefore, 
\begin{equation*}
A=\begin{bmatrix}
	0 & 0 & 1 & 1 \\
	H(0)z_1 & H(0)z_2 & 0 & 0 \\
	\dot H(0)z_1 & \dot H(0)z_2 & -H(0)^2z_1^2 & -H(0)^2z_2^2 \\
	\left(\ddot H(0)-H(0)^3z_1^2\right)z_1 & \left(\ddot H(0)-H(0)^3z_2^2\right)z_2 & -3H(0)\dot H(0)z_1^2 & -3H(0)\dot H(0)z_2^2
\end{bmatrix}.
\end{equation*}
In particular,
\begin{equation*}
\begin{split}
	\det A &=H(0)z_1\left(-3H(0)\dot H(0)^2z_2^3+ H(0)^2\left(\ddot H(0)-H(0)^3z_2^2\right)z_2^3\right)\\
	&\quad- H(0)z_2\left(-3H(0)\dot H(0)^2z_1z_2^2+ H(0)^2\left(\ddot H(0)-H(0)^3z_1^2\right)z_1z_2^2\right)\\
	&\quad-H(0)z_1\left(-3H(0)\dot H(0)^2z_1^2z_2+ H(0)^2\left(\ddot H(0)-H(0)^3z_2^2\right)z_1^2z_2\right)\\
	&\quad+ H(0)z_2\left(-3H(0)\dot H(0)^2z_1^3+ H(0)^2\left(\ddot H(0)-H(0)^3z_1^2\right)z_1^3\right)\\
	&=-H(0)^6z_1z_2\left(z_1^2-z_2^2\right)^2.
\end{split}
\end{equation*}
Since $H(0),z_1,z_2\neq 0$ and $z_1\neq z_2$, we conclude that $\det A\neq 0$. Finally, assume that $\mu=\mu_0$ and $u_1,\ldots,u_4$ are the functions given by \eqref{soluzionicasocritico}. Recall that 
\begin{equation*}
u_1=\sin \alpha \theta,\qquad
u_2=\cos\alpha\theta,\qquad
u_3= \theta\sin\alpha \theta,\qquad
u_4=\theta\cos\alpha\theta.
\end{equation*}
Then,
\begin{equation*}
\dot u_1=\alpha H\cos\alpha\theta,\qquad\dot u_3=H\sin\alpha\theta+\alpha H\theta\cos\alpha\theta,\qquad \dot u_4 =H\cos\alpha\theta-\alpha H\theta\sin\alpha\theta.
\end{equation*}
Moreover,
\begin{align*}
\ddot u_1&=\alpha\dot H\cos\alpha\theta-\alpha^2H^2\sin\alpha\theta,\\
\ddot u_3&=\dot H\sin\alpha\theta+2\alpha H^2\cos\alpha\theta+\alpha\dot H\theta\cos\alpha\theta-\alpha^2 H^2\theta\sin\alpha\theta,\\
\ddot u_4&=\dot H\cos\alpha\theta-2\alpha H^2\sin\alpha\theta-\alpha\dot H\theta\sin\alpha\theta-\alpha^2H^2\theta\cos\alpha\theta.
\end{align*}
Finally,
\begin{align*}
\dddot u_1(0)&=\ddot H(0)\alpha-H(0)^3\alpha^3,\\
\dddot u_3(0)&=6H(0)\dot H(0)\alpha,\\
\dddot u_4(0)&=\ddot H(0)-3H(0)^3\alpha^2.
\end{align*}
Therefore, 
\begin{equation*}
A=\begin{bmatrix}
	0 & 1 & 0 & 0 \\
	H(0)\alpha & \dot u_2(0) & 0 & H(0) \\
	\dot H(0)\alpha & \ddot u_2(0) & 2H(0)^2\alpha & \dot H(0) \\
	\ddot H(0)\alpha-H(0)^3\alpha^3 & \dddot u_2(0) & 6H(0)\dot H(0)\alpha & \ddot H(0)-3H(0)^3\alpha^2
\end{bmatrix}.
\end{equation*}
In particular,
\begin{equation*}
\begin{split}
	\det A&=-H(0)\alpha\left(2H(0)^2\ddot H(0)\alpha-6H(0)^5\alpha^3-6H(0)\dot H(0)^2\alpha\right)\\
	&\quad-H(0)\left(6H(0)\dot H(0)^2\alpha^2-2H(0)^2\ddot H(0)\alpha^2+2H(0)^5\alpha^4\right)\\
	&=4 H(0)^6\alpha^4.
\end{split}
\end{equation*}
Since $H(0),\alpha\neq 0$, we conclude that $\det A\neq 0$.
\end{proof}
As we will see, the study of stability requires to solve \eqref{primavoltaequazionenonhom} only when $\nu=0$. Nevertheless, we find a particular solution to \eqref{secondtimenonhom} for an arbitrary $\nu\in\rr$ for the sake of completeness.
\begin{proposition}
The following facts hold.
\begin{enumerate}
\item [(i)]If $\mu\neq 1$, then 
\begin{equation*}
	u_\nu(s)=\frac{\nu}{(2-\lambda)(1-\mu)}H(s)^{-1}
\end{equation*}
solves \eqref{secondtimenonhom}.
\item [(ii)] If $\mu=1$, then
\begin{equation*}
	u_\nu(s)=\frac{\nu}{2\lambda}\theta(s)(s-L)
\end{equation*}
solves \eqref{secondtimenonhom}.

\end{enumerate}
\end{proposition}
\begin{proof}
Assume first $\mu\neq 1$. For any $A\in\rr$, set $u_A=A H^{-1}$. Then
\begin{equation*}
\begin{split}
\left(r \ddot u_A\right)''+\mu \left(q \dot u_A\right)' + pu_A&=A\left( H^{-3}\left(2H^{-1}\right)''\right)''+A\mu\left(H^{-1}\left(2H^{-1}\right)'\right)'+A(2-\lambda)\\
&=A\left( H^{-3}\left(H\left(H^{-2}\right)'\right)'\right)''+A\mu\left(H^{-2}\right)''+A(2-\lambda)\\
&=A\left( H^{-2}\left(H^{-2}\right)''+H^{-3}H'\left(H^{-2}\right)'\right)''+A\mu(\lambda-2)+A(2-\lambda)\\
&=A\left( (\lambda-2)H^{-2}-\frac{1}{2}\left(\left(H^{-2}\right)'\right)^2\right)''+A(2-\lambda)(1-\mu)\\
&=A\left( (\lambda-2)\left(H^{-2}\right)'-\left(H^{-2}\right)'\left(H^{-2}\right)''\right)'+A(2-\lambda)(1-\mu)\\
&=A(2-\lambda)(1-\mu),
\end{split}
\end{equation*}
whence $u_A$ satisfies \eqref{secondtimenonhom} provided that $A=\frac{\nu}{(2-\lambda)(1-\mu)}$.  Assume instead $\mu= 1$. For any $\alpha\in\rr$, set $u_A=A (s-L)\theta$. 
Recall that 
\begin{equation}\label{inproofhmenodueprimo}
\left(H^{-2}\right)'=-\frac{2\sigma^2}{\pi^2}(s-L)=-\frac{2 x_0}{ L+x_0 }(s-L)=-(2-\lambda)(s-L),
\end{equation}
where $\sigma$ is as in \eqref{keg}.  Then
\begin{equation*}
\begin{split}
\left(r \ddot u_A\right)''+& \left(q \dot u_A\right)' + pu_A=2A\left( H^{-3}\left((s-L)\theta\right)''\right)''+2A\left(H^{-1}\left((s-L)\theta\right)'\right)'+A(2-\lambda)(s-L)\theta H\\
&=2A\left( H^{-3}\left(\theta+(s-L)H\right)'\right)''+2A\left(H^{-1}\theta+s-L\right)'+A(2-\lambda)(s-L)\theta H\\
&=2A\left( 2H^{-2}+(s-L)H^{-3}\dot H\right)''+4A-2A\theta H^{-2}\dot H+A(2-\lambda)(s-L)\theta H\\
&=A\left( 4H^{-2}-(s-L)\left(H^{-2}\right)'\right)''+4A+A\theta H\left(H^{-2}\right)'+A(2-\lambda)(s-L)\theta H\\
&\overset{\eqref{inproofhmenodueprimo}}{=}A\left( 4H^{-2}-(s-L)\left(H^{-2}\right)'\right)''+4A\\
&=A\left( 3\left(H^{-2}\right)'+(2-\lambda)(s-L)\right)'+4A\\
&\overset{\eqref{inproofhmenodueprimo}}{=}2A\lambda,
\end{split}
\end{equation*}
whence $u_A$ satisfies \eqref{secondtimenonhom} provided that $A=\frac{\nu}{2\lambda}$.
\end{proof}
\subsection{Stability} \label{subsec:stab}
Let $\gamma=\gamma(x_0,L)$ be the unique critical point of \Cref{uniquecriticalthm}. We are in position to prove that $\gamma$ is a stable critical point of $\F-\lambda\A$.
We recall once more that, if  $\mathcal V$ is a two-sided admissible variation of $\gamma$ and if $\varphi$ is as in \eqref{quattrfunzioni}, then $\dot\varphi(0)=\dot\varphi(2L)=0$. Therefore, following \Cref{examplestability}, we solve \eqref{minprobstab} when \begin{equation}\label{defiw1}
\mathcal W_1=\left\{\varphi\in W^{2,2}(0,2L)\,:\,\varphi(0)=\varphi(2L)=0,\,\dot\varphi(0)=\dot\varphi(2L)=0\right\}.
\end{equation} 
Notice that $\mathcal W_1$ is a closed subspace of $\mathcal S$ satisfying \eqref{contieneleamedianulla}. The minimum $\mu_{\mathcal W_1}$, as well as a corresponding minimizer, can be found as follows.
\begin{proposition}\label{characteigentwosided}
Let $\mu_{\mathcal W_1}$ be as in \eqref{minprobstab}. Then     $\mu_{\mathcal W_1}$ is the minimal $\mu> 0$ with the property that
\begin{equation}\label{bvptwosided}
\left\{
\begin{aligned}
&\left(r \ddot u\right)''+\mu \left(q \dot u\right)' + pu = 0\qquad\text{on }(0,2L), \\
&u(0) = u(2L) = 0, \\
&\dot u (0) = \dot u(2L) = 0
\end{aligned}
\right.
\end{equation}
admits a non-trivial solution. Moreover, any solution to \eqref{bvptwosided} with $\mu=\mu_{\mathcal W_1}$ is a minimizer of \eqref{minprobstab}.
\end{proposition}
\begin{proof}
By \Cref{exminstab}, \eqref{minprobstab} has a minimizer, say $u\in\mathcal W_1$. In particular, $u(0)=u(2L)=0$ and $\dot u(0)=\dot u(2L)=0$. Moreover, since $C_c^\infty(0,2L)\subseteq \mathcal W_1$, $u$ solves \eqref{primavoltaequazionenonhom} with $\nu=0$. Therefore, $u$ solves \eqref{bvptwosided} with $\mu=\mu_{\mathcal W_1}$. Assume by contradiction that there exists $0\leq \tilde \mu<\mathcal \mu_{\mathcal W_1}$ and $\tilde u\in\mathcal W_1$ such that $\tilde u$ solves \eqref{bvptwosided} with $\mu=\tilde \mu$.
Multiplying the first equation by $\tilde u $, integrating over $(0,2L)$ and integrating by parts, we get that 
\begin{equation*}
b(\tilde u,\tilde u)-\tilde \mu a(\tilde u,\tilde u)=0,
\end{equation*}
but then $\tilde \mu\leq \mu_{\mathcal W_1}$, a contradiction. Similarly, any solution to \eqref{bvptwosided} with $\mu=\mu_{\mathcal W_1}$ minimizes \eqref{minprobstab}.
\end{proof}
\begin{theorem}\label{teoremastabilitaperg}
Fix $0<3x_0<L$.
Let $\gamma=\gamma(x_0,L)$ be the unique critical point of \Cref{uniquecriticalthm}. 
Then $\mu_{\mathcal W_1}=\mu_{\mathcal W_1}\left(\frac{x_0}{L+x_0}\right)>1$. In particular, there exists $\mathcal C_{\mathcal W_1}=\mathcal C_{\mathcal W_1}\left(\frac{x_0}{L+x_0}\right)>0$ such that 
\begin{equation} \label{eq:est_secvarG}
\int_0^{2L}\left((2-\lambda)H\varphi^2-\frac{2\dot\varphi^2}{H}+\frac{2\ddot\varphi^2}{H^3}\right)\,ds\geq\mathcal C_{\mathcal W_1}\|\varphi\|^2_{W^{2,2}(0,2L)}
\end{equation}
for any $\varphi\in\mathcal W_1$, so that $\gamma$ is a stable critical point of $\F-\lambda\A$.
\end{theorem}
\begin{proof}
By \Cref{characteigentwosided}, $\mu_{\mathcal W_1}$ is minimal with the property that \eqref{bvptwosided} has a solution. Denote such solution by $u$.
First we show that $\mu_{\mathcal W_1}>\mu_0$.
Assume not by contradiction. Then $\mu_{\mathcal W_1}$ satisfies either $\mu_{\mathcal W_1}<\mu_0$ or $\mu_{\mathcal W_1}=\mu_0$. Assume first that  $\mu_{\mathcal W_1}<\mu_0$. Therefore, by \Cref{soluzgenrralestab}, there exists $a,b,c,d\in\rr$ such that 
\begin{equation*}
u=a\cos \alpha\theta\sinh\beta\theta+b\sin \alpha\theta\cosh\beta\theta+c\sin \alpha\theta\sinh\beta\theta+d\cos \alpha\theta\cosh\beta\theta,
\end{equation*}
where $\alpha$ and $\beta$ are as in \eqref{alfabetacomplesso}. Since $\theta(0)=0$ and $u(0)=0,$ then $d=0$, so that
\begin{align*}
\dot u&=H\big[-a\alpha\sin\alpha\theta\sinh\beta\theta+a\beta\cos\alpha\theta\cosh\beta\theta+b\alpha\cos\alpha\theta\cosh\beta\theta\\
&\quad+b\beta\sin\alpha\theta\sinh\beta\theta+c\alpha\cos\alpha\theta\sinh\beta\theta+c\beta\sin\alpha\theta\cosh\beta\theta\big].
\end{align*}
Since $\dot u(0)=0$ and $H(0)\neq 0$, then $a\beta+b\alpha=0$. Since $\alpha\neq 0$, then $b=-\frac{a\beta}{\alpha}$, whence
\begin{align*}
u&=a\left(\cos \alpha\theta\sinh\beta\theta-\frac{\beta}{\alpha}\sin \alpha\theta\cosh\beta\theta\right)+c\sin \alpha\theta\sinh\beta\theta,\\
\dot u&=aH\left(-\left(\alpha+\frac{\beta^2}{\alpha}\right)\sin\alpha\theta\sinh\beta\theta\right)+cH\left(\alpha\cos\alpha\theta\sinh\beta\theta+\beta\sin\alpha\theta\cosh\beta\theta\right).
\end{align*}
Moreover, as $u(2L)=\dot u(2L)=0$, $\theta(2L)=2\pi$, $H(0)\neq 0$ and $\alpha\neq 0$, then 
\begin{equation*}
\begin{cases}
a\left(\alpha\cos 2\pi\alpha\sinh2\pi\beta-\beta\sin 2\pi\alpha\cosh2\pi\beta\right)+c\alpha \sin 2\pi\alpha\sinh 2\pi\beta&=0\vspace{0.3cm}\\ 
a\left(-\left(\alpha+\frac{\beta^2}{\alpha}\right)\sin 2\pi\alpha\sinh 2\pi\beta\right)+c\left(\alpha\cos 2\pi\alpha\sinh 2\pi\beta+\beta\sin2\pi\alpha\cosh 2\pi\beta\right)&=0.
\end{cases}
\end{equation*}
The above is a linear, homogeneous system in the variables $a$ and $c$. We denote its coefficient matrix by $M$. Since $u\neq0$, then $(a,b)\neq(0,0)$. Therefore, the determinant of $M$ vanishes. But 
\begin{equation*}
\begin{split}
\det M&=\alpha^2\cos^22\pi\alpha\sinh^22\pi\beta-\beta^2\sin^2 2\pi\alpha\cosh^22\pi\beta+\left(\alpha^2+\beta^2\right)\sin^22\pi\alpha\sinh^22\pi\beta\\
&=\alpha^2\sinh^22\pi\beta-\beta^2\sin^2 2\pi\alpha.
\end{split}
\end{equation*}
Therefore, as $\alpha,\beta\neq 0$, $\det M=0$ if and only if
\begin{equation*}
\left(\frac{\sin 2\pi\alpha}{\alpha}\right)^2=\left(\frac{\sinh 2\pi\beta}{\beta}\right)^2.
\end{equation*}
However, this is impossible, because, as $\alpha,\beta>0$,
\begin{equation*}
\left(\frac{\sin 2\pi\alpha}{\alpha}\right)^2<4\pi^2\qquad\text{and}\qquad \left(\frac{\sinh 2\pi\beta}{\beta}\right)^2>4\pi^2,
\end{equation*}
and so $\mu_{\mathcal W_1}\geq\mu_0$.
Instead, assume that  $\mu_{\mathcal W_1}=\mu_0$.  By \Cref{soluzgenrralestab}, there exist $a,b,c,d\in\rr$ such that 
\begin{equation*}
u=a\sin \alpha\theta+b\cos \alpha\theta+c\theta \sin \alpha\theta+d\theta\cos \alpha\theta.
\end{equation*}
Since $\theta(0)=0$ and $u(0)=0,$ then $b=0$, so that
\begin{equation*}
\dot u=a\alpha H\cos\alpha\theta+cH\sin\alpha\theta+c\alpha H\theta\cos\alpha\theta+d H\cos\alpha\theta-d\alpha H\theta\sin\alpha\theta.
\end{equation*}
Since $\dot u(0)=0$ and $\theta(0)=0$, then $a\alpha H(0)+d H(0)=0$. But $H(0)\neq 0$, whence $d=-a\alpha$, and 
\begin{align*}
u&=a\left(\sin \alpha\theta-\alpha\theta\cos\alpha\theta\right)+c\theta \sin \alpha\theta,\\
\dot u&=a\alpha^2H\theta\sin\alpha\theta+cH\left(\sin\alpha\theta+\alpha\theta\cos\alpha\theta\right).
\end{align*}
As $u(2L)=\dot u(2L)=0$, $\theta(2L)=2\pi$ and $H(0)\neq 0$, then 
\begin{equation*}
\begin{cases}
a\left(\sin 2\pi\alpha-2\pi\alpha\cos2\pi\alpha\right)+c\left(2\pi\sin 2\pi\alpha\right)&=0,\\
a\left(2\pi\alpha^2\sin 2\pi\alpha\right)+c\left(\sin2\pi\alpha+2\pi\alpha\cos 2\pi\alpha\right)&=0.
\end{cases}
\end{equation*}
Again, if M is the coefficient matrix of the above system,  then $\det M=0$. But
\begin{equation*}
\det M=\sin^2 2\pi\alpha-4\pi^2\alpha^2\cos^2 2\pi\alpha-4\pi^2\alpha^2\sin^2 2\pi\alpha=\sin^2 2\pi\alpha-4\pi^2\alpha^2.
\end{equation*}
But since $\alpha\neq 0$, $\det M=0$ if and only if 
\begin{equation*}
\left(\frac{\sin 2\pi\alpha}{2\pi\alpha}\right)^2=1,
\end{equation*}
which is impossible.
Therefore, $\mu_{\mathcal W_1}>\mu_0$. 
By \Cref{soluzgenrralestab}, there are $a,b,c,d\in\rr$ such that
\begin{equation*}
u=a\sin z_1\theta+b\sin z_2\theta +c\cos z_1\theta+d\cos z_2\theta,
\end{equation*}
where $z_1=\alpha-\gamma$, $z_2=\alpha+\gamma$ and $\gamma $ is as in \eqref{defigamma}. Since $u(0)=0$ and $\theta(0)=0$, then $d=-c$, whence
\begin{equation*}
\dot u=H\left(az_1\cos z_1\theta+bz_2\cos z_2\theta -c z_1\sin z_1\theta+cz_2\sin z_2\theta\right).
\end{equation*}
By $\dot u(0)=0$ and $H(0)\neq 0$, $az_1+bz_2=0$. Since $z_2\neq 0$, $b=-\frac{az_1}{z_2}$, whence
\begin{align*}
u&=a\left(\sin z_1\theta-\frac{z_1}{z_2}\sin z_2\theta\right)+c\left(\cos z_1\theta-\cos z_2\theta\right),\\
\dot u&=aH\left(z_1\cos z_1\theta-z_1\cos z_2\theta\right)+cH\left(-z1\sin z_1\theta+z_2\sin z_2\theta\right).
\end{align*}
As $u(2L)=\dot u(2L)=0$, $\theta(2L)=2\pi$, then 
\begin{equation*}
\begin{cases}
a\left(z_2\sin 2\pi z_1-z_1\sin 2\pi z_2\right)+c\left(z_2\cos 2\pi z_1-z_2\cos 2\pi z_2\right)&=0,\\
a\left(z_1\cos 2 \pi z_1-z_1\cos 2 \pi z_2\right)+c\left(-z_1\sin 2 \pi z_1+z_2\sin 2 \pi z_2\right)&=0.
\end{cases}
\end{equation*}
Again, $\det M=0$. But
\begin{equation*}
\begin{split}
\det M&=-z_1z_2\sin^2 2\pi z_1+z_2^2\sin 2\pi z_1\sin 2\pi z_2+z_1^2\sin 2\pi z_1\sin 2\pi z_2-z_1z_2\sin^2 2\pi z_2\\
&-z_1z_2 \cos^2 2\pi z_1+ 2z_1z_2\cos 2\pi z_1\cos 2 \pi z_2-z_1z_2\cos ^22\pi z_2\\
&=2z_1z_2\left (\cos 2\pi z_1\cos 2 \pi z_2-1\right)+ \left(z_1^2+z_2^2\right)\sin 2\pi z_1\sin 2\pi z_2\\
&=\left(\alpha^2-\gamma^2\right)\left(\cos 4\pi\alpha+\cos 4\pi\gamma-2\right)+\left(\alpha^2+\gamma^2\right)\left(\cos 4\pi\gamma-\cos 4\pi \alpha\right)\\
&=-2\gamma^2\cos 4\pi\alpha+2\alpha^2\cos 4\pi\gamma-2\alpha^2+2\gamma^2\\
&=2\gamma^2\left(1-\cos 4\pi\alpha\right)-2\alpha^2\left(1-\cos 4\pi\gamma\right)\\
&=4\gamma^2\sin^22\pi\alpha-4\alpha^2\sin^22\pi\gamma.
\end{split}
\end{equation*}
Therefore $\det M=0$ if and only if 
\begin{equation}\label{contrastart}
\left(\frac{\sin 2\pi\alpha}{\alpha}\right)^2=\left(\frac{\sin 2\pi\gamma}{\gamma}\right)^2.
\end{equation}
Assume by contradiction that $\mu_{\mathcal W_1}\leq 1$. Notice that, since $\mu>\mu_0$,
\begin{equation}\label{auxmuwquasifine}
\mu_{\mathcal W_1}-\frac{x_0}{L+x_0}>2\left(-\frac{x_0}{L+x_0}+\sqrt{\frac{x_0}{L+x_0}}\right)\geq 0,
\end{equation}
where the last inequality follows since $\frac{x_0}{L+x_0}\leq 1$. Therefore
\begin{equation*}
\begin{split}
(\alpha+\gamma)^2&=\alpha^2+\gamma^2+2\alpha\gamma\\
&=\frac{1}{4}\left(2\mu_{\mathcal W_1}+\frac{2x_0}{L+x_0}+2\sqrt{\left(\mu_{\mathcal W_1}+\frac{x_0}{L+x_0}\right)^2-\frac{4x_0}{L+x_0}}\right)\\
&=\frac{1}{4}\left(2\mu_{\mathcal W_1}+\frac{2x_0}{L+x_0}+2\sqrt{\left(\mu_{\mathcal W_1}-\frac{x_0}{L+x_0}\right)^2}\right)\\
&\overset{\eqref{auxmuwquasifine}}{=}\mu_{\mathcal W_1}.
\end{split}
\end{equation*}
In particular, since $\alpha,\gamma>0$, then $  \alpha+\gamma=\sqrt{\mu_{\mathcal W_1}}\leq 1$. Moreover, as $\gamma<\alpha$, then $\gamma\in\left(0,\frac{1}{2}\right).$ Set $f(x)=\left(\frac{\sin2\pi x}{x}\right)^2$. Notice that $f'(x)=\left(\frac{2\sin2\pi x}{x}\right)\left(\frac{2\pi x\cos 2\pi x-\sin 2\pi x}{x^2}\right)$. But $2\pi x\cos 2\pi x-\sin 2\pi x<0$ on $\left(0,\frac{1}{2}\right)$, whence $f$ is strictly decreasing, and hence injective, on $\left(0,\frac{1}{2}\right]$. In particular, $\alpha\in \left(\frac{1}{2},1\right)$, since otherwise the injectivity of $f$ on $\left(0,\frac{1}{2}\right]$ and the fact that $\alpha\neq \gamma$ would contradict \eqref{contrastart}. Since $\alpha\in \left(\frac{1}{2},1\right)$, then $1-\alpha\in \left(0,\frac{1}{2}\right)$, and moreover, since $\alpha+\gamma\leq 1$, then $1-\alpha\geq\gamma$. Finally, $\alpha\in \left(\frac{1}{2},1\right)$ implies $\frac{1}{\alpha^2}<\frac{1}{(1-\alpha)^2}$. In conclusion, 
\begin{equation*}
\left(\frac{\sin 2\pi\alpha}{\alpha}\right)^2 <\left(\frac{\sin 2\pi\alpha}{1-\alpha}\right)^2=\left(\frac{\sin 2\pi(1-\alpha)}{1-\alpha}\right)^2\leq  \left(\frac{\sin 2\pi\gamma}{\gamma}\right)^2,
\end{equation*}
a contradiction with \eqref{contrastart}. Therefore, $\mu_{\mathcal W_1}>1$. The rest of the thesis follows by \Cref{corollstbofg}.
\end{proof}

\subsection{One-sided area-preserving stability}
As already pointed out, we do not enter in the case described in \Cref{onesidedexample}. We limit to detect the relevant boundary value problem associated with it.
\begin{proposition}
Let
\begin{equation*}
\mathcal W_2=\left\{\varphi\in W^{2,2}(0,2L)\,:\,\varphi(0)=\varphi(2L)=0,\,\int_0^{2L}\varphi=0\right\}.
\end{equation*}
Let $\mu_{\mathcal W_2}$ be as in \eqref{minprobstab}. Then     $\mu_{\mathcal W_22}$ is the minimal $\mu> 0$ with the property that
\begin{equation}\label{bvponesidedF}
\left\{
\begin{aligned}
&\left(r \ddot u\right)''+\mu \left(q \dot u\right)' + pu = \nu\qquad\text{on }(0,2L), \\
&u(0) = u(2L) = 0, \\
&\ddot u (0) = \ddot u(2L) = 0\\
&\int_0^{2L}u\,ds=0.
\end{aligned}
\right.
\end{equation}
admits a non-trivial solution for some $\nu\in\rr$. Moreover, any solution to \eqref{bvponesidedF} with $\mu=\mu_{\mathcal W_2}$ is a minimizer of \eqref{minprobstab}.
\end{proposition}
\begin{proof}
By \Cref{exminstab}, \eqref{minprobstab} has a minimizer, say $u\in\mathcal W_22$. In particular, $u(0)=u(2L)=0$ and $\int_0^{2L}u\,ds=0$. Moreover, recall that there exists $\nu\in\rr$ such that $u$ solves \eqref{primavoltaequazionenonhom}.
Finally, by \eqref{weakstabwithbc} and \eqref{primavoltaequazionenonhom}, and recalling that $r(0)=r(2L)>0$,
\begin{equation*}
\ddot u(0)\dot v(0)=  \ddot u(2L)\dot v(2L)\qquad\text{for any $v\in\mathcal W$.}
\end{equation*}
In particular, $\ddot u(0)=\ddot u(2L)=0$.
Therefore, $u$ solves \eqref{bvptwosided} with $\mu=\mu_{\mathcal W_2}$. The rest of the proof follows as in the proof of \Cref{characteigentwosided}.
\end{proof}

\section{Local minimality}\label{sec_locmin}
This section is dedicated to show that given $x_0, L > 0$ with $3x_0 < L$, the critical point $\gamma(x_0,L)$ found in \cref{scoprirepunticritici} is a local minimizer of $\F-\lambda\A$ with respect to normal geodesic variations. The proof of this fact strongly relies on the lower bound for the second variation of $\F-\lambda\A$ given in \cref{characteigentwosided}.
We recall that 
\begin{equation*}
\mathcal W_1=\left\{\varphi\in W^{2,2}(0,2L)\,:\,\varphi(0)=\varphi(2L)=0,\,\dot\varphi(0)=\dot\varphi(2L)=0\right\}.
\end{equation*} 


\begin{theorem}\label{locminthm}
Fix  $x_0, \cA_0 > 0$ with $\frac{3}{2}\pi x_0^2<\cA_0 $. Let $\ga= \ga(x_0,L(x_0,\cA_0))$ be the critical point given by \Cref{uniquecriticalthm}. Let $\lambda=\lambda(x_0,L(x_0,\cA_0))$ be the constant given in \eqref{definitionoflambda}. Let $\mathcal W_1$ be as in \eqref{defiw1}. Then, there exists $\varepsilon=\varepsilon(x_0,\A_0)>0$ such that for all $\phi \in \cW \cap C^2[0,2L]$ with $\phi\not\equiv 0$ and
\begin{equation*}
\| \phi \|_{C^{2}(0,2L)} \leq \varepsilon ,
\end{equation*}
it holds that
\begin{equation*}
(\mathcal F-\lambda  \A )(\gamma) < (\mathcal F-\lambda  \A )(\gamma + \phi N) , 
\end{equation*}
that is, $\gamma$ is a minimizer for $\F-\lambda\A$ among small normal geodesic variations of $\gamma$.
In particular, if $\cA(\gamma) = \cA(\gamma + \phi N)$, then
\begin{equation*}
\mathcal{F}(\gamma) < \mathcal{F}(\gamma + \phi N) ,
\end{equation*}
that is, $\gamma$ is an area-preserving minimizer of $\mathcal{F}$
\end{theorem}
\begin{proof}
Fix $\phi \in \cW \cap C^2[0,2L]$, and set $\cV(s,t) = \gamma(s) + t \, \phi(s)N(s)$.
As usual, denote by $\gamma_t = \cV(\cdot,t) = (x_t,y_t)$ and by $H_t$ the curvature of $\gamma_t$. Consider the Lagrangian function
\begin{equation*}
\Lag(q,z,r)= \dfrac{(z_1^2 + z_2^2)^2}{z_1 r_2 - z_2 r_1} - \dfrac{\lambda}{2} (q_1 z_2 - q_2 z_1)
\qquad \text{where} \qquad q = (q_1,q_2), \, z = (z_1, z_2) , \, r = (r_1, r_2) .
\end{equation*}
By the definition of $\F-\lambda\A$, we have
\begin{equation*}
(\mathcal F-\lambda  \A )(\gamma_t) = \int_0^{2L} \Lag(\gamma_t(s), \dot \gamma_t(s) , \ddot \gamma_t(s)) \, ds .
\end{equation*}
Let $q(s)$, $r(s)$ and $z(s)$ be such that
\begin{equation} \label{eq:gamma_t}
\gamma_t(s) = \gamma(s) + t \, q(s),\qquad \dot \gamma_t(s) = \dot \gamma(s) + t \, r(s) \qquad \text{and} \qquad \ddot \gamma_t(s) = \ddot \gamma(s) + t \, z(s) .
\end{equation}
From now on, we will omit the dependence on $s$ for the functions that appear along the proof, when this does not create confusion. By the definition of $\gamma_t$, $q = \phi N$. Then a simple computation allows us to find the components of $r$ and $z$:
\begin{equation} \label{eq:comp_p_dotp}
\begin{split}
& r_1 = \dot \phi \dot y + \phi \ddot y \hspace{1.9cm} r_{2} = - \dot \phi \dot x - \phi \ddot x \\
& z_{1} = \ddot \phi \dot y + 2 \dot \phi \ddot y + \phi \dddot y \qquad z_2 = - \ddot \phi \dot x - 2 \dot \phi \ddot x - \phi \dddot x .
\end{split}
\end{equation}
Denote by $U := (q,r,z)$. Thanks to \eqref{eq:gamma_t}, we can apply the Taylor expansion formula with Lagrange remainder to $\Lag(\gamma_t , \dot \gamma_t , \ddot \gamma_t)$, for $t=1$, deducing that, for a suitable $\tau = \tau(s) \in (0,1)$,
\begin{equation} \label{eq:Lag_t=1}
\begin{split}
\Lag(\gamma_1, \dot \gamma_1 , \ddot \gamma_1) & = \Lag(\gamma, \dot \gamma , \ddot \gamma) + \left \langle \nabla \Lag(\gamma, \dot \gamma , \ddot \gamma), U \right \rangle + \dfrac{1}{2} \left \langle \nabla^2 \Lag(\gamma_{\tau}, \dot \gamma_{\tau}, \ddot \gamma_{\tau}) U, U \right \rangle \\
& = \Lag(\gamma, \dot \gamma , \ddot \gamma) + \left \langle \nabla \Lag(\gamma, \dot \gamma , \ddot \gamma), U \right \rangle + \dfrac{1}{2} \left \langle \nabla^2 \Lag(\gamma, \dot \gamma, \ddot \gamma) U, U \right \rangle + \langle R_{\tau} U, U \rangle , 
\end{split}
\end{equation}
where $R_{\tau} := \frac{1}{2} \left( \nabla^2 \Lag(\gamma_{\tau}, \dot \gamma_{\tau}, \ddot \gamma_{\tau}) - \nabla^2 \Lag(\gamma, \dot \gamma, \ddot \gamma) \right)$.
We now observe that
\begin{equation} \label{eq:first_var_Lag}
\int_0^{2L} \left \langle \nabla \Lag(\gamma, \dot \gamma , \ddot \gamma), U \right \rangle ds = 0 ,
\end{equation}
because $\gamma$ is critical for $\mathcal F-\lambda  \A $, and, by the expression for the second variation of $\F-\lambda\A$ at $\gamma$ provided by \eqref{secvarG}, we also deduce that
\begin{equation} \label{eq:sec_var_Lag}
\dfrac{1}{2} \int_0^{2L} \left \langle \nabla^2 \Lag(\gamma, \dot \gamma, \ddot \gamma) U, U \right \rangle ds \geq \int_0^{2L}\left((2-\lambda)H\varphi^2-\frac{2\dot\varphi^2}{H}+\frac{2\ddot\varphi^2}{H^3}\right)\,ds.
\end{equation}
Therefore, integrating \eqref{eq:Lag_t=1} and applying \eqref{eq:first_var_Lag} and \eqref{eq:sec_var_Lag}, we infer that
\begin{equation} \label{eq:est_Ggamma1}
(\mathcal F-\lambda  \A )(\gamma_1) \geq (\mathcal F-\lambda  \A )(\gamma) + \int_0^{2L}\left((2-\lambda)H\varphi^2-\frac{2\dot\varphi^2}{H}+\frac{2\ddot\varphi^2}{H^3}\right)\,ds + \int_0^{2L} \langle R_{\tau} U, U \rangle \, ds .
\end{equation}
We estimate the second integral on the right hand side of \eqref{eq:est_Ggamma1}. First we observe that 
\begin{equation} \label{eq:CS_estim}
\left | \langle R_{\tau} U , U \rangle \right | \leq \left | R_{\tau} U \right | \left | U \right | \leq \left | R_{\tau} \right | \left | U \right |^2 .
\end{equation}
Moreover, by the definition of $U$, there exists $C_1 = C_1(x_0,\cA_0) > 0$ with the property that, for every $s$,
\begin{equation} \label{eq:pw_est_U}
|U(s)| \leq C_1 \sup \{ |\phi(s)|, |\dot \phi(s)|, |\ddot \phi(s)| \} , 
\end{equation}
and so in particular, for some $C_2 = C_2(x_0,\A_0) > 0$,
\begin{equation} \label{eq:est_unif_U}
\|U\|_{C^0[0,2L]} \leq C_2 \| \phi \|_{C^2[0,2L]} \leq C_2 \, \epsilon .
\end{equation}
Consider
$K := \{ (\gamma_t(s), \dot \gamma_t(s), \ddot \gamma_t(s)) \, : \, s \in [0,2L] \, \, \text{and} \, \, 0 \leq t \leq 1 \}$.
Recalling that $(\gamma(s), \dot \gamma(s), \ddot \gamma(s)) + t \, U(s)$, by the estimate \eqref{eq:est_unif_U}, if we take $\epsilon > 0$ small enough, we can choose $C_3 = C_3(x_0,\cA_0) > 0$ such that
\begin{equation*}
\| \nabla^3 \Lag \|_{C^0(K)} \leq C_3 .
\end{equation*}
Hence, we have
\begin{equation} \label{eq:pw_est_Rtau}
|R_{\tau}(s)| = \dfrac{1}{2} |\nabla^2 \Lag(\gamma_{\tau}(s), \dot \gamma_{\tau}(s), \ddot \gamma_{\tau}(s)) - \nabla^2 \Lag(\gamma(s), \dot \gamma(s), \ddot \gamma(s))| \leq \dfrac{C_3}{2} |U(s)| . 
\end{equation}
Thus, \eqref{eq:CS_estim} in combination with \eqref{eq:pw_est_Rtau} and \eqref{eq:pw_est_U} ensure that, for some $C_4 = C_4(x_0,\cA_0) > 0$,
\begin{equation} \label{eq:est_R_def}
\begin{split}
\int_0^{2L} \langle R_{\tau} U, U \rangle \, ds \leq \dfrac{C_1^3 \, C_3}{2} \int_0^{2L} \sup \{ |\phi|, |\dot \phi|, |\ddot \phi| \}^3 ds 
\leq C_4 \, \| \phi \|_{C^2[0,2L]} \, \| \phi \|_{W^{2,2}(0,2L)}^2 . 
\end{split}
\end{equation}
Expression \eqref{eq:est_R_def} together with \eqref{eq:est_secvarG} gives that
\begin{equation*} 
\begin{split}
\int_0^{2L}\left((2-\lambda)H\varphi^2-\frac{2\dot\varphi^2}{H}+\frac{2\ddot\varphi^2}{H^3}\right)\,ds + \int_0^{2L} R_{\tau}(s) \, ds & \geq \mathcal{C}_{\cW} \| \phi \|_{W^{2,2}(0,2L)}^2 \left( 1 - \dfrac{C_4}{\mathcal{C}_{\cW}} \| \phi \|_{C^2[0,2L]} \right) ,
\end{split}
\end{equation*}
and consequently, if we pick $\varepsilon$ sufficiently small,
\begin{equation} \label{eq:positivity_secvar}
\begin{split}
\int_0^{2L}\left((2-\lambda)H\varphi^2-\frac{2\dot\varphi^2}{H}+\frac{2\ddot\varphi^2}{H^3}\right)\,ds + \int_0^{2L} R_{\tau}(s) \, ds & \geq \dfrac{\mathcal{C}_{\cW}}{2} \| \phi \|_{W^{2,2}(0,2L)}^2 \geq 0.
\end{split}
\end{equation}
Loading \eqref{eq:positivity_secvar} inside \eqref{eq:est_Ggamma1} we finally conclude that
\begin{equation*}
(\mathcal F-\lambda  \A )(\gamma) \leq (\mathcal F-\lambda  \A )(\gamma_1) = (\mathcal F-\lambda  \A )(\gamma + \phi N) \qquad \text{whenever} \qquad \| \phi \|_{C^2[0,2L]} \leq \varepsilon .
\end{equation*}
Clearly, if $(\mathcal F-\lambda  \A )(\gamma) = (\mathcal F-\lambda  \A )(\gamma + \phi N)$ for some $\phi$ satisfying $\| \phi \|_{C^2[0,2L]} \leq \varepsilon$, then \eqref{eq:est_Ggamma1} and \eqref{eq:positivity_secvar} imply that $\| \phi \|_{W^{2,2}(0,2L)} = 0$, i.e., $\phi \equiv 0$.
\end{proof}

\bibliographystyle{abbrv}
\bibliography{biblio}
\end{document}